\documentclass[12pt,leqno]{amsart}
\usepackage{enumerate}
\usepackage{amsrefs}
\usepackage{amsfonts, amsmath, amssymb, amscd, amsthm, bm, cancel}
\usepackage{url}
\usepackage{graphicx}
\usepackage[
linktocpage=true,colorlinks,citecolor=magenta,linkcolor=blue,urlcolor=magenta]{hyperref}
\usepackage{multicol}
\usepackage{comment}
\usepackage[margin=1in]{geometry}
\makeatletter
\@namedef{subjclassname@2020}{\textup{2020} Mathematics Subject Classification}
\makeatother
\theoremstyle{plain} \newtheorem{thm}{Theorem}[section]
\theoremstyle{plain} \newtheorem{prop}[thm]{Proposition}
\theoremstyle{plain} \newtheorem{prob}[thm]{Problem}
\theoremstyle{definition} 
\theoremstyle{plain} \newtheorem{cor}[thm]{Corollary}
\theoremstyle{plain} \newtheorem{lem}[thm]{Lemma}
\theoremstyle{plain} 
\theoremstyle{remark} \newtheorem{rmk}[thm]{Remark}
\theoremstyle{plain} \newtheorem{conj}{Conjecture}
 
 \newtheorem*{ack}{Acknowledgments}

 \numberwithin{equation}{section}
 \allowdisplaybreaks

\newcommand{\bR}{\mathbb{R}}

\newcommand{\bS}{\mathbb{S}}
\newcommand{\bB}{\mathbb{B}}
\newcommand{\cK}{\mathcal{K}}
\newcommand{\cH}{\mathcal{H}}

\newcommand{\tr}{\mathrm{tr}}

\newcommand{\dist}{\mathrm{dist}}
\newcommand{\divg}{\mathrm{div}}
\newcommand{\Id}{\mathrm{Id}}
\newcommand{\II}{\mathrm{II}}

\newcommand{\tV}{\widetilde V}
\newcommand{\tC}{\widetilde C}

\newcommand{\dd}{\,d}
\newcommand{\R}{\mathbb{R}}
\newcommand{\HS}{\mathrm{HS}}
\begin{document}

\title[Dual quermassintegrals beyond the volume index]{On the Brunn-Minkowski inequality for $q$-th dual quermassintegrals with $q>n$}

\author{Haizhong Li}
\address{Department of Mathematical Sciences, Tsinghua University, Beijing 100084, P.R. China}
\email{\href{mailto:lihz@tsinghua.edu.cn}{lihz@tsinghua.edu.cn}}

\author{Yao Wan}
\address{Department of Mathematics, The Chinese University of Hong Kong, Shatin, N.T., Hong Kong}
\email{\href{mailto:yaowan@cuhk.edu.hk}{yaowan@cuhk.edu.hk}}

\author{Yijia Zhang}
\address{Department of Mathematical Sciences, Tsinghua University, Beijing 100084, P.R. China}
\email{\href{mailto:zhang-yj23@mails.tsinghua.edu.cn}{zhang-yj23@mails.tsinghua.edu.cn}}



\begin{abstract}
In this paper, we study the Brunn-Minkowski inequality for $q$-th dual quermassintegrals with $q>n$. This problem was recently posed by Sadovsky and Zhang. First, by a second variation argument and a dimension reduction construction, we show that the inequality fails for arbitrary convex bodies when $q>n$, and fails even in the origin-symmetric class when $q>n+2$. Secondly, we prove the endpoint case $q=n+2$ for origin-symmetric convex bodies via Hadwiger's inequality for the polar moment of inertia. Finally, for unconditional convex bodies, we establish the inequality in the full range $0<q\le n+1$ by using a singular weighted Reilly formula and a coordinate-slice Hardy inequality. As applications, we derive several uniqueness results for the corresponding dual curvature measures.
\end{abstract}

\maketitle

\section{Introduction}\label{sec-1}

The Brunn-Minkowski inequality is a cornerstone of convex geometry. 
It originated from the work of Brunn and Minkowski, and it later became the starting point of the classical Brunn-Minkowski theory. In its classical form, the Brunn-Minkowski 
inequality states that for convex bodies $K,L\subset \bR^n$,
\begin{align}\label{classical-BM}
    |K+L|^{\frac1n}\geq |K|^{\frac1n}+|L|^{\frac1n},
\end{align}
where $K+L$ is the Minkowski sum. Equality holds if and only if $K$ and $L$ are homothetic.

The dual Brunn-Minkowski theory, initiated by Lutwak \cite{Lut75} in the 1970s, is based on the radial addition of star bodies and on dual mixed volumes. It has become an important part of convex geometry, especially through its connections with intersection bodies, the Busemann-Petty problem, and geometric tomography; see, e.g., \cite{Lut88,Gar06,Kold05,Zha99}.
If $K_1,\ldots,K_n\subset \bR^n$ are convex bodies containing the origin in their interiors, their dual mixed volume is defined by
\begin{align}\label{def-dual-mixed-volume}
    \widetilde V(K_1,\ldots,K_n)
    =\frac1n\int_{\bS^{n-1}}\rho_{K_1}(u)\cdots \rho_{K_n}(u)\dd u,
\end{align}
where $\mathbb{S}^{n-1}$ is the $(n-1)$-dimensional unit sphere, and $\rho_K(u)=\max\{\lambda>0:\lambda u\in K\}$ is the radial function of $K$.
Taking $K_1=\cdots=K_j=K$ and $K_{j+1}=\cdots=K_n=\bB^n_2$ gives the $j$-th dual quermassintegral
\[
    \widetilde V_j(K)=\frac1n\int_{\bS^{n-1}}\rho_K(u)^j\dd u,
    \qquad j=0,1,\ldots,n.
\]

Note that usual quermassintegrals can be represented by averages of projection volumes, whereas dual quermassintegrals can be represented by averages of section volumes,
\[
    \widetilde V_j(K)
    =
    \frac{\omega_n}{\omega_j}
    \int_{G_{n,j}} |K\cap \xi|_j\,d\xi,
    \qquad j=1,\ldots,n.
\]
In recent years, the dual Brunn-Minkowski theory has developed further through the introduction of dual curvature measures and their Minkowski problems by Huang, Lutwak, Yang, and Zhang \cite{HLYZ16}; see also the survey \cite{HYZ25}.

For any real number $q$, the index of the dual quermassintegral can be extended by setting
\begin{align}\label{def-dual-q}
    \tV_q(K)=\frac1n\int_{\bS^{n-1}}\rho_K(u)^q\dd u.
\end{align}
When $q>0$, polar coordinates give the equivalent formula
\begin{align}\label{polar-dual-q}
    \tV_q(K)=\frac qn\int_K |x|^{q-n}\dd x.
\end{align}
Thus $\tV_q$ is $q$-homogeneous, i.e. $\tV_q(tK)=t^q\tV_q(K)$ for $t>0$. The case $q=n$ is the ordinary volume, namely $\tV_n(K)=|K|$.

It is therefore natural to ask whether a Brunn-Minkowski inequality holds for $\tV_q$ with respect to Minkowski addition
\begin{align}\label{BM-dual}
    \tV_q(K+L)^{\frac1q}\geq \tV_q(K)^{\frac1q}+\tV_q(L)^{\frac1q}.
\end{align}
When $q=n$, this is exactly the classical Brunn-Minkowski inequality 
\eqref{classical-BM}. For arbitrary convex bodies containing the origin in their interiors, Xi and Zhang \cite{XZ22} proved \eqref{BM-dual} in the range $0<q\leq 1$. More generally, they established $L_p$ Brunn-Minkowski inequalities for dual quermassintegrals in the range $p\geq q$, and obtained corresponding uniqueness results for $(p,q)$-dual curvature measures.

For origin-symmetric convex bodies, Lutwak conjectured that \eqref{BM-dual} holds for the integer range $2\leq q\leq n-1$. This conjecture was later stated in the framework of the dual Minkowski problem; see, for example, \cite{HLYZ16,HYZ25}. Recently, Sadovsky and Zhang \cite{SZ25}, building on the analytic works of Kolesnikov-Milman \cite{KM18,KM22}, 
Kolesnikov-Livshyts \cite{KL21}, and Cordero-Erausquin-Rotem \cite{CER23}, 
proved the following strengthened form.

\begin{thm}[\cite{SZ25}]\label{thm-SZ}
Let $K,L\subset \bR^n$ be origin-symmetric convex bodies. Then for $0<q\leq n$,
\begin{align}\label{SZ-BM}
    \tV_q(K+L)^{\frac1q}\geq \tV_q(K)^{\frac1q}+\tV_q(L)^{\frac1q}.
\end{align}
Equality holds if and only if $K$ and $L$ are dilates of each other.
\end{thm}

Although the functional $\tV_q$ is well-defined for all real $q$, the proof in \cite{SZ25} uses essentially the condition $q\leq n$. This led Sadovsky and Zhang to ask the following question.

\begin{prob}[\cite{SZ25}, Problem 12]\label{problem-SZ}
Does the inequality \eqref{BM-dual} hold when $q>n$?
\end{prob}

Some partial negative results were obtained by Xu and Yang \cite{XY26}. They proved that \eqref{BM-dual} fails for some origin-symmetric convex bodies when $q>2n+1$, and they also studied the local behavior near the Euclidean ball.

The purpose of this paper is to investigate Problem \ref{problem-SZ}. Our first result shows that, without the origin-symmetry assumption, the range $0<q\leq n$ is sharp.

\begin{thm}\label{thm-nonsym-fail}
Let $q>n$. Then there exist smooth strictly convex bodies $K,L\subset \bR^n$ containing the origin in their interiors such that
\begin{align}
    \tV_q(K+L)^{\frac1q}<\tV_q(K)^{\frac1q}+\tV_q(L)^{\frac1q}.
\end{align}
\end{thm}

In the origin-symmetric class, we obtain a stronger counterexample range than that in \cite{XY26}.

\begin{thm}\label{thm-sym-fail}
Let $q>n+2$. Then there exist smooth strictly convex origin-symmetric bodies $K,L\subset \bR^n$ such that
\begin{align}
    \tV_q(K+L)^{\frac1q}<\tV_q(K)^{\frac1q}+\tV_q(L)^{\frac1q}.
\end{align}
\end{thm}

The proof is based on a dimension reduction argument. Degenerating an $n$-dimensional body to a product of a planar rectangle and a thin cube reduces the $n$-dimensional parameter $q$ to a planar parameter $Q=q-n+2$. The planar rectangle calculation gives a strict failure for $Q>4$, which is precisely $q>n+2$.

At the endpoint $q=n+2$, we prove a positive result for origin-symmetric bodies. This follows from a classical inequality of Hadwiger \cite{Had56} for the polar moment of inertia.

\begin{thm}\label{thm-endpoint}
Let $K,L\subset \bR^n$ be origin-symmetric convex bodies. Then
\begin{align}\label{endpoint-BM}
    \tV_{n+2}(K+L)^{\frac1{n+2}}\geq \tV_{n+2}(K)^{\frac1{n+2}}+\tV_{n+2}(L)^{\frac1{n+2}}.
\end{align}
Equality holds if and only if $K$ and $L$ are dilates of each other.
\end{thm}

Motivated by this, we propose the following conjecture.

\begin{conj}\label{conj-sym-range}
Let $K,L\subset \bR^n$ be origin-symmetric convex bodies. Then for $n<q<n+2$,
\begin{align}
    \tV_q(K+L)^{\frac1q}\geq \tV_q(K)^{\frac1q}+\tV_q(L)^{\frac1q}.
\end{align}
Equality holds if and only if $K$ and $L$ are dilates of each other.
\end{conj}

Finally, for unconditional convex bodies, Hu-Ivaki \cite{HI26} proved the endpoint case \(q=n+2\), as well as a positive range \(q\in (n,n+3-2\sqrt{2}]\). In this paper, we verify the above conjecture in the unconditional case. In fact, we obtain the full range $0<q\le n+1$ for unconditional convex bodies.

\begin{thm}\label{thm-unconditional}
Let $K,L\subset \bR^n$ be unconditional convex bodies. Then for $0<q\leq n+1$,
\begin{align}\label{unc-BM}
    \tV_q(K+L)^{\frac1q}\geq \tV_q(K)^{\frac1q}+\tV_q(L)^{\frac1q}.
\end{align}
Moreover, if in addition $K$ and $L$ are of class $C^2_+$, then equality holds if and only if $K$ and $L$ are dilates of each other.
\end{thm}

\begin{rmk}
In the unconditional $C^2_+$ setting, the equality characterization in
Theorem~\ref{thm-unconditional} answers Problem 9 of Sadovsky-Zhang \cite{SZ25}, which asks for the equality case in \eqref{SZ-BM}. 
\end{rmk}

The proof uses a singular weighted Reilly formula for the density $|x|^\alpha$, where $\alpha=q-n\in(-n,1]$, together with a Hardy-type inequality on coordinate slices adapted to the sign of $\alpha$.
The argument is inspired by Hu-Ivaki \cite{HI26}, but is adapted here to the full range $0<q\le n+1$.
As an application, we derive a Minkowski inequality and consequent uniqueness for the $(1,q)$-th dual curvature measure in the unconditional $C^2_+$ class.

The paper is organized as follows. In Section~\ref{sec-2} we recall the basic definitions and variational formulas. Sections~\ref{sec-3} and~\ref{sec-4} contain counterexamples: the inequality fails for arbitrary convex bodies when $q>n$, and fails for origin‑symmetric bodies when $q>n+2$ by reducing to a planar calculation. Section~\ref{sec-5} treats the endpoint $q=n+2$, where Hadwiger's inequality for the polar moment of inertia gives the Brunn–Minkowski inequality for all symmetric bodies. The core of the paper lies in Sections~\ref{sec-unconditional-prep}-\ref{sec-unconditional-proof}, where we develop a singular weighted Reilly formula and a coordinate‑slice Hardy inequality, valid for both positive and negative indices, and prove the full range $0<q\le n+1$. Finally, Section~\ref{sec-8} derives uniqueness results for $(1,q)$-th dual curvature measures as consequences of our Brunn–Minkowski inequalities.

\section{Preliminaries}\label{sec-2}

\subsection{Convex bodies and dual quermassintegrals}\

Let $\cK_o^n$ denote the class of convex bodies in $\bR^n$ containing the origin in their interiors. We write $\cK_{+}^n$ for the subclass of convex bodies with $C^2$ boundary and positive Gauss curvature, and $\cK_{+,e}^n$ for the class of origin-symmetric convex bodies in $\cK_{+}^n$.

For $K\in \cK_o^n$, the support function of $K$ is defined by
\[
    h_K(u)=\max_{x\in K}\langle x,u\rangle,\qquad u\in \bS^{n-1}.
\]
If $K,L\in \cK_o^n$ and $a,b\geq 0$, then
\[
    h_{aK+bL}=ah_K+bh_L.
\]
We also recall that the polar body of $K\in\cK_o^n$ is
\[
K^\ast=\{x\in \bR^n:\langle x,y\rangle\le 1\ \text{for all }y\in K\},
\]
and that $K^\ast\in \cK_o^n$ whenever $K\in \cK_o^n$.

For $q>0$, define the weighted measure
\begin{align}\label{mu-q}
    d\mu_q(x)=|x|^{q-n}\,dx,\qquad 
    \mu_q(K)=\int_K |x|^{q-n}\,dx.
\end{align}
By the polar-coordinate formula,
\[
    \tV_q(K)=\frac qn \mu_q(K).
\]
In particular, since $\mu_q$ is homogeneous of degree $q$, the Brunn-Minkowski inequality \eqref{BM-dual} for $\tV_q$ is equivalent to the corresponding inequality for $\mu_q$.

It will be convenient to use the standard notation
\[
W(x)=(n-q)\log |x|,\qquad d\mu_q=e^{-W}\,dx.
\]
Although $W$ is singular at the origin, this does not affect the variational arguments below, since all convex bodies under consideration contain a fixed neighborhood of the origin.

\subsection{Weighted second variational formula}\

Let $K\in \cK_+^n$ and consider a perturbation of the support function
\[
    h_t=h_K+t f,
\]
where $f\in C^2(\bS^{n-1})$ and $|t|$ is sufficiently small. Let $K_t$ be the corresponding convex body. Writing $\nu$ for the Gauss map of $\partial K$, the normal velocity of $\partial K_t$ at $t=0$ is $ \phi=f\circ\nu$.

Let $\II$ be the second fundamental form of $\partial K$, and define the weighted mean curvature
\[
    H_W=H-\langle \nabla W,\nu\rangle,
\]
where $H=\tr \II$. The following variational formula is due to Kolesnikov-Milman \cite{KM18}; see also \cite[Eq.~(10)]{SZ25}.

\begin{lem}[Weighted first and second variation]\label{lem-variation}
For the above perturbation,
\begin{align}
    \left.\frac{d}{dt}\right|_{t=0}\mu_q(K_t)
    &=\int_{\partial K}\phi\,d\mu_{\partial K,q},\\
    \left.\frac{d^2}{dt^2}\right|_{t=0}\mu_q(K_t)
    &=\int_{\partial K}\left(H_W\phi^2-\left\langle \II^{-1}\nabla_{\partial K}\phi,\nabla_{\partial K}\phi\right\rangle\right)d\mu_{\partial K,q},
\end{align}
where
\[
    d\mu_{\partial K,q}=|x|^{q-n}d\cH^{n-1}_{\partial K}.
\]
\end{lem}

Consequently, the local $1/q$-concavity of $\mu_q$ at $K$ is equivalent to
\begin{align}\label{local-concavity}
    \mu_q(K)\left.\frac{d^2}{dt^2}\right|_{t=0}\mu_q(K_t)
    \leq \frac{q-1}{q}
    \left(\left.\frac{d}{dt}\right|_{t=0}\mu_q(K_t)\right)^2.
\end{align}

The next lemma summarizes the equivalences that will be used repeatedly below.

\begin{lem}\label{lem-equivalent}
Let $q>0$. For origin-symmetric convex bodies, the following are equivalent:
\begin{enumerate}[(i)]
    \item The Brunn-Minkowski inequality for the dual quermassintegral $\tV_q$ holds:
    \[
        \tV_q(K+L)^{1/q}\ge \tV_q(K)^{1/q}+\tV_q(L)^{1/q}
    \]
    for all origin-symmetric convex bodies $K,L\subset\bR^n$.

    \item The measure $\mu_q$ is $1/q$-concave on origin-symmetric convex bodies, i.e.,
    \[
        \mu_q((1-\lambda)K+\lambda L)^{1/q}
        \ge (1-\lambda)\mu_q(K)^{1/q}+\lambda\mu_q(L)^{1/q}
    \]
    for all $\lambda\in[0,1]$ and all origin-symmetric convex bodies $K,L\subset\bR^n$.

    \item For every $K\in\cK_{+,e}^n$ and every $f\in C_e^2(\bS^{n-1})$ such that
    $h_t=h_K+t f$ is a smooth support-function perturbation, the second variation inequality holds:
    \begin{align}\label{eq-secondvar-equiv}
        \mu_q(K)\left.\frac{d^2}{dt^2}\right|_{t=0}\mu_q(K_t)
        \le \frac{q-1}{q}
        \left(\left.\frac{d}{dt}\right|_{t=0}\mu_q(K_t)\right)^2.
    \end{align}

    \item Equivalently, writing $\phi=f\circ\nu$, one has
    \begin{align}\label{eq-expanded-equiv}
        \mu_q(K)\int_{\partial K}\left(
        H_W\phi^2-\left\langle \II^{-1}\nabla_{\partial K}\phi,\nabla_{\partial K}\phi\right\rangle
        \right)d\mu_{\partial K,q}
        \le \frac{q-1}{q}
        \left(\int_{\partial K}\phi\,d\mu_{\partial K,q}\right)^2.
    \end{align}
\end{enumerate}
Moreover, the above equivalences also hold for general convex bodies in $\mathcal{K}_o^n$, where the restriction to even functions is replaced by general functions in $C^2(\bS^{n-1})$.
\end{lem}

\subsection{Unconditional bodies}\

A convex body $K\subset \bR^n$ is called unconditional if it is invariant under reflections in the coordinate hyperplanes. Equivalently, its support function satisfies
\[
h_K(\sigma_1 u_1,\dots,\sigma_n u_n)=h_K(u_1,\dots,u_n)
\]
for every choice of signs $\sigma_i\in\{\pm1\}$. For an unconditional convex body, its intersection with each orthant is determined by its part in the positive orthant; see, e.g., \cite{BL95,Sar15,HI26}.

In particular, if $K$ is unconditional and $x\in K$, then every $y\in \bR^n$ satisfying $|y_i|\le |x_i|$ for all $i$ also belongs to $K$. This monotonicity property will be used repeatedly in the unconditional case.

\section{Failure for arbitrary convex bodies when $q>n$}\label{sec-3}

In this section, we prove Theorem~\ref{thm-nonsym-fail}. The argument is based on testing the local $1/q$-concavity of $\mu_q$ at the Euclidean ball.

Let $K=B_2^n$. On $\partial K=\bS^{n-1}$, we have
\[
    \II=\Id,\qquad H=n-1,\qquad \nu=x,\qquad |x|=1.
\]
Moreover, for the density $d\mu_q=|x|^{q-n}\dd x$,
\[
    W(x)=(n-q)\log|x|,\qquad \nabla W=(n-q)x,
\]
and hence
\[
    H_W=H-\langle \nabla W,\nu\rangle=(n-1)-(n-q)=q-1.
\]

Therefore, by Lemma \ref{lem-variation}, \eqref{local-concavity} is equivalent to
\begin{equation}\label{ball-quadratic}
    \int_{\bS^{n-1}}|\nabla_{\bS^{n-1}}\phi|^2\,\dd u
    -(q-1)\int_{\bS^{n-1}}(\phi-\bar\phi)^2\,\dd u\ge 0,
\end{equation}
where $\bar\phi$ is given by
\[
    \bar\phi=\frac{1}{|\bS^{n-1}|}\int_{\bS^{n-1}}\phi\,\dd u.
\]

Recall that the first nonzero eigenvalue of $-\Delta_{\bS^{n-1}}$ is $n-1$, and its eigenspace consists of the restrictions of linear functions. Hence \eqref{ball-quadratic} holds for all smooth $\phi\in C^2(\bS^{n-1})$ if and only if $ q-1\le n-1$, that is, $q\le n.$

\begin{proof}[Proof of Theorem~\ref{thm-nonsym-fail}]
Assume $q>n$. Choose a nonzero first spherical harmonic $\phi(u)=\langle u,e_1\rangle$. Then
\[
    \int_{\bS^{n-1}}|\nabla_{\bS^{n-1}}\phi|^2\,\dd u
    =(n-1)\int_{\bS^{n-1}}\phi^2\,\dd u
    <(q-1)\int_{\bS^{n-1}}\phi^2\,\dd u.
\]
Thus the second variation of $\mu_q(K_t)^{1/q}$ at the ball in the direction $\phi$ is positive, so $t\mapsto \mu_q(K_t)^{1/q}$ is locally convex at $t=0$. Consequently, for sufficiently small $t>0$,
\[
    \mu_q\!\left(\frac{K_t+K_{-t}}{2}\right)^{1/q}
    <
    \frac12\mu_q(K_t)^{1/q}+\frac12\mu_q(K_{-t})^{1/q}.
\]
By the $q$-homogeneity of $\mu_q$, this is equivalent to
\[
    \mu_q(K_t+K_{-t})^{1/q}
    <
    \mu_q(K_t)^{1/q}+\mu_q(K_{-t})^{1/q}.
\]
Using \eqref{polar-dual-q}, we obtain the failure of the Brunn-Minkowski inequality for $\tV_q$. This completes the proof.
\end{proof}

\begin{rmk}
The perturbation above is odd, so it does not preserve origin-symmetry and only gives counterexamples for arbitrary convex bodies.

In the origin-symmetric class, one must use even perturbations. Then the first nonconstant even eigenvalue of $-\Delta_{\bS^{n-1}}$ is $2n$, so the Euclidean ball is locally stable up to $q\le 2n+1$. This local threshold was identified by Xu and Yang \cite{XY26}, but it is not the global sharp threshold.
\end{rmk}

\section{Symmetric counterexamples for $q>n+2$}\label{sec-4}

In this section, we prove Theorem~\ref{thm-sym-fail} by a dimension reduction argument. Throughout this section, we write $\mu_q^{(n)}$ for the $q$-homogeneous measure on $\R^n$ with density
\[
    d\mu_q^{(n)}(x)=|x|^{q-n}\dd x.
\]

\subsection{A planar rectangle calculation}\

Let $Q>0$ and consider the planar weighted measure
\[
    d\mu_Q^{(2)}(x,y)=(x^2+y^2)^{\frac{Q-2}{2}}\,\dd x\,\dd y.
\]
For small $t$, define the origin-symmetric rectangles
\[
    R_t=[-1-t,1+t]\times[-1+t,1-t],
\]
and set
\[
    F_Q(t)=\mu_Q^{(2)}(R_t).
\]
Then
\[
    \frac{R_t+R_{-t}}{2}=R_0.
\]

\begin{lem}\label{lem-planar-rectangle}
If $Q>4$, then $F_Q'(0)=0$ and $F_Q''(0)>0$. In particular, $F_Q(t)^{1/Q}$ is locally convex at $t=0$.
\end{lem}

\begin{proof}
Write $\alpha=\frac{Q-2}{2}$. Then
\[
    F_Q(t)=\int_{-1-t}^{1+t}\int_{-1+t}^{1-t}(x^2+y^2)^\alpha\,\dd y\,\dd x.
\]

It is convenient to define $a(t)=1+t,\, b(t)=1-t$, and
\[
    G(a,b)=4\int_0^a\int_0^b (x^2+y^2)^\alpha\,\dd y\,\dd x,
\]
so that $F_Q(t)=G(a(t),b(t))$. By differentiating with respect to the side lengths, we get
\[
    \partial_a G(a,b)=4\int_0^b (a^2+y^2)^\alpha\,\dd y,\qquad
    \partial_b G(a,b)=4\int_0^a (x^2+b^2)^\alpha\,\dd x,
\]
\[
    \partial_{aa}G(a,b)=8\alpha a\int_0^b(a^2+y^2)^{\alpha-1}\dd y,\quad
    \partial_{bb}G(a,b)=8\alpha b\int_0^a(x^2+b^2)^{\alpha-1}\dd x,
\]
and
\[
    \partial_{ab}G(a,b)=4(a^2+b^2)^\alpha.
\]
At $t=0$, since $a(0)=b(0)=1$, it follows that
\[
     F_Q'(0)=\partial_aG(1,1)-\partial_bG(1,1)=0,
\]
and
\begin{align*}
    F_Q''(0)
    &=\partial_{aa}G(1,1)-2\partial_{ab}G(1,1)+\partial_{bb}G(1,1)
    = 8\Big(\alpha\int_{-1}^1(1+y^2)^{\alpha-1}\dd y-2^\alpha\Big).
\end{align*}
Therefore,
\begin{equation}\label{planar-condition}
    F_Q''(0)>0
    \quad\Longleftrightarrow\quad
    \alpha\int_{-1}^1(1+y^2)^{\alpha-1}\dd y>2^\alpha.
\end{equation}

Assume now that $Q>4$, equivalently $\alpha>1$. Since $1+y^2\ge 2|y|$ for $|y|\le 1$, with strict inequality except at $|y|=1$, we obtain
\[
    \alpha\int_{-1}^1(1+y^2)^{\alpha-1}\dd y
    >2\alpha\int_0^1(2y)^{\alpha-1}\dd y
    =2^\alpha.
\]
Hence \eqref{planar-condition} holds, so $F_Q''(0)>0$.
Combining with $F_Q'(0)=0$ and $F_Q(0)>0$,
\[
    \left.\frac{\dd^2}{\dd t^2}\right|_{t=0}F_Q(t)^{1/Q}
    =\frac{1}{Q}F_Q(0)^{1/Q-1}F_Q''(0)>0,
\]
which implies that $t\mapsto F_Q(t)^{1/Q}$ is locally convex at $t=0$.
\end{proof}

\subsection{Dimension reduction}\

Let $n\ge2$ and set $d=n-2$. For a planar origin-symmetric convex body $A\subset\R^2$, define
\[
    A_\varepsilon=A\times[-\varepsilon,\varepsilon]^d\subset\R^n.
\]
Let $ Q=q-n+2$. For $x=(y,z)\in\R^2\times\R^d$, we have $|x|^{q-n}=(|y|^2+|z|^2)^{\frac{Q-2}{2}}.$
Therefore
\[
    \mu_q^{(n)}(A_\varepsilon)
    =\int_A\int_{[-\varepsilon,\varepsilon]^d}(|y|^2+|z|^2)^{\frac{Q-2}{2}}\,\dd z\,\dd y.
\]
As $\varepsilon\to 0$, we have
\[
    \int_{[-\varepsilon,\varepsilon]^d}(|y|^2+|z|^2)^{\frac{Q-2}{2}}\,\dd z
    =(2\varepsilon)^d|y|^{Q-2}+o(\varepsilon^d),
\]
uniformly for $y$ in compact subsets of $\R^2$. Integrating in $y$ yields
\begin{equation}\label{dimension-reduction}
    \mu_q^{(n)}(A_\varepsilon)
    =(2\varepsilon)^d\mu_Q^{(2)}(A)+o(\varepsilon^d).
\end{equation}

For the rectangle family above, the same argument applies uniformly in $t$ near $0$. 
Hence, the asymptotic \eqref{dimension-reduction} also holds uniformly together with the first and second derivatives in $t$.

\begin{proof}[Proof of Theorem~\ref{thm-sym-fail}]
Assume $q>n+2$, and denote $Q=q-n+2>4.$ Consider
\[
    K_{\varepsilon,t}=R_t\times[-\varepsilon,\varepsilon]^{n-2}.
\]
By \eqref{dimension-reduction}, we have
\[
    \mu_q^{(n)}(K_{\varepsilon,t})
    =(2\varepsilon)^{n-2}F_Q(t)+o(\varepsilon^{n-2}),
\]
uniformly near $t=0$, together with the first and second derivatives in $t$. 
Lemma~\ref{lem-planar-rectangle} implies that, for all sufficiently small $\varepsilon>0$,
\[
    \left.\frac{\dd^2}{\dd t^2}\right|_{t=0}\mu_q^{(n)}(K_{\varepsilon,t})^{1/q}>0.
\]
Thus $t\mapsto \mu_q^{(n)}(K_{\varepsilon,t})^{1/q}$ is locally convex at $t=0$, and therefore for some sufficiently small nonzero $t$,
\[
    \mu_q^{(n)}(K_{\varepsilon,0})^{1/q}
    <
    \frac12\mu_q^{(n)}(K_{\varepsilon,t})^{1/q}
    +\frac12\mu_q^{(n)}(K_{\varepsilon,-t})^{1/q}.
\]
Since $\frac{1}{2}K_{\varepsilon,t}+\frac{1}{2}K_{\varepsilon,-t}=K_{\varepsilon,0}$, we obtain
\[
    \mu_q^{(n)}(K_{\varepsilon,t}+K_{\varepsilon,-t})^{1/q}
    <
    \mu_q^{(n)}(K_{\varepsilon,t})^{1/q}
    +\mu_q^{(n)}(K_{\varepsilon,-t})^{1/q}.
\]
Thus the inequality fails for origin-symmetric boxes. By smoothing the support functions, the strict inequality persists under Hausdorff approximation by smooth strictly convex origin-symmetric bodies. The proof is complete.
\end{proof}

\begin{rmk}
The theorem reduces the unknown origin-symmetric range in Problem~\ref{problem-SZ} to
\[
    n<q\le n+2.
\]
The endpoint $q=n+2$ will be treated in the next section.
\end{rmk}
\subsection{An \(L_p\) Brunn-Minkowski threshold for \(q>q_p^{(n)}\)}\label{sec-Lp-threshold} \

For \(0<p\le1\), \(0<\lambda<1\), and origin-symmetric convex bodies \(K,L\) containing the origin, define the \(L_p\) geometric combination
\[
(1-\lambda) \cdot K+_{p} \lambda\cdot L
:=\Bigl[\bigl((1-\lambda)h_K^p+\lambda h_L^p\bigr)^{1/p}\Bigr],
\]
and for \(p=0\) the logarithmic combination
\[
(1-\lambda)\cdot K+_{0} \lambda \cdot L:=\bigl[h_K^{1-\lambda}h_L^\lambda\bigr].
\]

Let \(R_t=[-1-t,1+t]\times[-1+t,1-t]\) be the rectangles. A direct calculation gives
\[
\frac{1}{2}\cdot R_t+_p\frac{1}{2}\cdot R_{-t}=r_{p,t}R_0,
\]
where
\[
r_{p,t}
=
\begin{cases}
\Bigl(\frac{(1+t)^p+(1-t)^p}{2}\Bigr)^{1/p}, & 0<p\le1,\\[2mm]
\sqrt{1-t^2}, & p=0.
\end{cases}
\]
For \(Q>2\), define
\[
P(Q)
=
(Q-2)\frac{\displaystyle\int_0^1(1-s^2)(1+s^2)^{(Q-4)/2}\,ds}
{\displaystyle\int_0^1(1+s^2)^{(Q-2)/2}\,ds}.
\]
The integrand is strictly increasing in \(Q\), hence \(P\) is continuous, strictly increasing on \([2,4]\), and satisfies \(P(2)=0\), \(P(4)=1\). Thus for every \(0\le p\le1\) there is a unique \(Q_p\in[2,4]\) such that \(P(Q_p)=p\).

\begin{prop}\label{prop-Lp-fail-threshold}
Let \(0\le p\le1\), and set
\[
q_p^{(n)}:=n+Q_p-2.
\]
If \(q>q_p^{(n)}\), then there exist origin-symmetric \(C^\infty_+\) convex bodies
\(K,L\subset\R^n\) such that, for \(0<p\le1\),
\[
\tV_q\bigl(\frac{1}{2}\cdot K+_p\frac{1}{2}\cdot L\bigr)^{\frac{p}{q}}
<
\frac12\tV_q(K)^{\frac{p}{q}}
+\frac12\tV_q(L)^{\frac{p}{q}},
\]
while for \(p=0\),
\[
\tV_q\bigl(\frac{1}{2}\cdot K+_0\frac{1}{2}\cdot L\bigr)
<
\tV_q(K)^{\frac12}\tV_q(L)^{\frac12}.
\]
Hence, the \(L_p\) Brunn-Minkowski inequality for \(\tV_q\) fails in
\(\mathcal{K}_e^n\) for every \(q>q_p^{(n)}\).
\end{prop}

\begin{proof}
We first prove the planar case \(n=2\). Write \(Q=q\); the condition
\(q>q_p^{(2)}\) is \(Q>Q_p\), equivalently \(p<P(Q)\). Let
\[
F_{p,Q}(t)=\tV_Q(R_t)-r_{p,t}^Q\tV_Q(R_0).
\]
By Lemma~\ref{lem-planar-rectangle} and \(R_t+_pR_{-t}=r_{p,t}R_0\), we obtain $F_{p,Q}'(0)=0$ and
\[
F_{p,Q}''(0)=4Q\bigl(P(Q)-p\bigr)
\int_0^1(1+s^2)^{(Q-2)/2}\,ds.
\]
Since \(p<P(Q)\), we have \(F_{p,Q}''(0)>0\). Hence for all sufficiently small
nonzero \(t\),
\[
\tV_Q(R_t)>r_{p,t}^Q\tV_Q(R_0)=\tV_Q(R_t+_pR_{-t}).
\]
Since \(R_{-t}\) is a rotation of \(R_t\),
\(\tV_Q(R_{-t})=\tV_Q(R_t)\). Therefore, for \(0<p\le1\),
\[
\tV_Q\bigl(\frac{1}{2}\cdot R_t+_p\frac{1}{2}\cdot R_{-t}\bigr)^{\frac{p}{Q}}
<
\tV_Q(R_t)^{\frac{p}{Q}}
=
\frac12\tV_Q(R_t)^{\frac{p}{Q}}
+\frac12\tV_Q(R_{-t})^{\frac{p}{Q}},
\]
while for \(p=0\),
\[
\tV_Q\bigl(\frac{1}{2}\cdot R_t+_0\frac{1}{2}\cdot R_{-t}\bigr)
<
\tV_Q(R_t)
=
\tV_Q(R_t)^{\frac12}\tV_Q(R_{-t})^{\frac12}.
\]
This proves the desired planar failure.

For \(n\ge2\), set \(Q=q-n+2\) and consider
\[
K_{t,\varepsilon}=R_t\times[-\varepsilon,\varepsilon]^{n-2}.
\]
The dimension reduction argument of Section~\ref{sec-4} gives
\[
\tV_q(K_{t,\varepsilon})
=
\frac{2q}{nQ}(2\varepsilon)^{n-2}\tV_Q(R_t)
+o(\varepsilon^{n-2}),
\]
and the \(L_p\) combination is compatible with the product structure:
\[
K_{t,\varepsilon}+_pK_{-t,\varepsilon}
=
(r_{p,t}R_0)\times[-\varepsilon,\varepsilon]^{n-2}.
\]
Thus
\[
\tV_q\bigl(\frac{1}{2}\cdot K_{t,\varepsilon}+_p\frac{1}{2}\cdot K_{-t,\varepsilon}\bigr)
=
\frac{2q}{nQ}(2\varepsilon)^{n-2}\tV_Q(r_{p,t}R_0)
+o(\varepsilon^{n-2}).
\]
The strict inequality in the planar case is preserved for sufficiently small
\(\varepsilon>0\). Smoothing the boxes yields the desired \(C^\infty_+\) bodies.
This completes the proof.
\end{proof}

\begin{rmk}
For \(p=0\), we have \(Q_0=2\) and hence \(q_0^{(n)}=n\). In this case, the threshold in Proposition~\ref{prop-Lp-fail-threshold} agrees with the threshold in Xiong and Yang \cite[Theorem 1.4]{XiongY26}.
For \(p=1\), \(Q_1=4\), we get \(q_1^{(n)}=n+2\), which is precisely the range covered by Theorem~\ref{thm-sym-fail}.
\end{rmk}

\section{The endpoint $q=n+2$ and Hadwiger's inequality}\label{sec-5}

For a convex body $K\subset \R^n$, let $g(K)=\frac{1}{|K|}\int_K x\,\dd x$ be its centroid. Its polar moment of inertia with respect to the centroid is
\[
    I(K)=\int_K |x-g(K)|^2\,\dd x.
\]
More generally, if $E=u^\perp$ is a hyperplane through the origin, where $u\in S^{n-1}$, define the planar moment of inertia
\[
    T_E(K)=\int_K \dist(x,E)^2\,\dd x=\int_K \langle x,u\rangle^2\,\dd x
\]
provided that the centroid of $K$ is at the origin. Hadwiger \cite{Had56} proved the Brunn-Minkowski inequality for the polar moment of inertia.

\begin{thm}[\cite{Had56}]\label{thm-Hadwiger}
For convex bodies $K,L\subset\R^n$,
\[
    I(K+L)^{\frac1{n+2}}\ge I(K)^{\frac1{n+2}}+I(L)^{\frac1{n+2}}.
\]
Equality holds if and only if $K$ and $L$ are homothetic.
\end{thm}

We include a brief proof for completeness.

\begin{proof}[Proof of Theorem~\ref{thm-Hadwiger}]
We divide the argument into three steps.

\smallskip
\noindent\emph{Step 1: one-sided planar moments.}
Fix $u\in \bS^{n-1}$ and write $H^+=\{x\in\bR^{n}:\langle x,u\rangle\ge0\}$. For a convex body $M\subset H^+$, define
\[
    T_+(M)=\int_M \langle x,u\rangle^2\,\dd x.
\]
Embed $\R^n$ into $\R^{n+2}=\R^n\times\R^2$ and set
\[
    M^*=\{(x,s,t):x\in M,\ 0\le s\le \langle x,u\rangle,\ 0\le t\le \langle x,u\rangle\}.
\]
Since $\langle x,u\rangle\ge0$ on $M$, the set $M^*$ is convex, and
\[
    |M^*|
    =\int_M\int_0^{\langle x,u\rangle}\int_0^{\langle x,u\rangle}\dd s\,\dd t\,\dd x
    =\int_M \langle x,u\rangle^2\,\dd x
    =T_+(M).
\]
Moreover, for convex bodies $M,N\subset H^+$ and $\alpha,\beta\ge0$,
\[
    (\alpha M+\beta N)^*=\alpha M^*+\beta N^*.
\]
 
Hence the Brunn-Minkowski inequality in $\R^{n+2}$ gives
\[
    T_+(\alpha M+\beta N)^{\frac1{n+2}}
    \ge \alpha T_+(M)^{\frac1{n+2}}+\beta T_+(N)^{\frac1{n+2}}.
\]

\smallskip
\noindent\emph{Step 2: planar moments.}
For $u\in \bS^{n-1}$ and a convex body $M$, let
\[
    T(M,u)=\int_M \langle x-g(M),u\rangle^2\,\dd x
\]
be the planar moment of inertia in direction $u$. We claim that
\begin{equation}\label{eq-planar-Hadwiger}
    T(\alpha K+\beta L,u)^{\frac1{n+2}}
    \ge \alpha T(K,u)^{\frac1{n+2}}+\beta T(L,u)^{\frac1{n+2}}
\end{equation}
for all convex bodies $K,L$ and all $\alpha,\beta\ge0$.

Set $S=\alpha K+\beta L$, and let $E=u^\perp+g(S)$ be the hyperplane orthogonal to $u$ through the centroid of $S$. Write $H^\pm$ for the two half-spaces bounded by $E$, and define
\[
    T_\pm^E(M)=\int_{M\cap H^\pm}\dist(x,E)^2\,\dd x.
\]
Then $T(S,u)=T_E(S)=T_+^E(S)+T_-^E(S)$. Note that for any convex body $M$,
\begin{equation}\label{TE-relation}
    T_E(M)=\int_{M}\langle x-g(S),u\rangle^2\,\dd x\geq \int_{M}\langle x-g(M),u\rangle^2\,\dd x=T(M,u).
\end{equation}

Now translate $K$ and $L$ in opposite $\pm u$-directions:
\[
    K_\tau=K+\tau\beta u,\qquad L_\tau=L-\tau\alpha u.
\]
Then $\alpha K_\tau+\beta L_\tau=S$ for all $\tau\in\R$.
As $\tau$ varies, $K_\tau$ and $L_\tau$ move in opposite directions across the fixed plane $E$, so the ratios $T_+^E(K_\tau)/T_-^E(K_\tau)$ and $T_+^E(L_\tau)/T_-^E(L_\tau)$ vary continuously, the former from $0$ to $\infty$ and the latter from $\infty$ to $0$. Hence for some $\tau_0\in\mathbb R$,
\[
\frac{T_+^E(K_{\tau_0})}{T_-^E(K_{\tau_0})}
=
\frac{T_+^E(L_{\tau_0})}{T_-^E(L_{\tau_0})}.
\]
Since $S\cap H^\pm \supset \alpha(K_{\tau_0}\cap H^\pm)+\beta(L_{\tau_0}\cap H^\pm)$,
Step~1 applied in each half-space gives
\[
    T_\pm^E(S)^{\frac1{n+2}}
    \ge \alpha T_\pm^E(K_{\tau_0})^{\frac1{n+2}}
      + \beta T_\pm^E(L_{\tau_0})^{\frac1{n+2}}.
\]
Hence
\[
 T(S,u)^{\frac1{n+2}}=T_E(S)^{\frac1{n+2}}
    \ge \alpha T_E(K_{\tau_0})^{\frac1{n+2}}
      + \beta T_E(L_{\tau_0})^{\frac1{n+2}}.
\]

Since $T(\cdot,u)$ is invariant under translation, by \eqref{TE-relation}, we obtain
\[
T_E(K_{\tau_0})\ge T(K_{\tau_0},u)=T(K,u), \qquad T_E(L_{\tau_0})\ge T(L_{\tau_0},u)=T(L,u).
\]
Therefore
\[
T(S,u)^{\frac1{n+2}}
\ge \alpha T(K,u)^{\frac1{n+2}}+\beta T(L,u)^{\frac1{n+2}},
\]
which proves \eqref{eq-planar-Hadwiger}.

\smallskip
\noindent\emph{Step 3: polar moments.}
By the homogeneity and translation invariance of \(I\), we may assume that $I(K)=I(L)=1$, and the centroids of \(K\) and \(L\) are at the origin.

For $u\in\bS^{n-1}$, define the quadratic form
\[f(u):=T_{u^\perp}(K)-T_{u^\perp}(L)=u^T(M_K-M_L)u,\]
where $M_K=\int_K xx^T\,\dd x$ and $M_L=\int_L xx^T\,\dd x$. Since $ \tr (M_K-M_L)=I(K)-I(L)=0$, there exists an orthonormal basis \( \{u_i\}_{i=1}^n\) such that
\[
    f(u_i)= u_i^T (M_K-M_L) u_i=0,\qquad i=1,\dots,n.
\]
Now apply \eqref{eq-planar-Hadwiger} to the hyperplanes $u_i^\perp$:
\[
    T_{u_i^\perp}(K+L)^{\frac1{n+2}}
    \ge T_{u_i^\perp}(K)^{\frac1{n+2}}+T_{u_i^\perp}(L)^{\frac1{n+2}}.
\]
Summing over $i$, using $T_{u_i^\perp}(K)=T_{u_i^\perp}(L)$ and $I(M)=\sum_{i=1}^n T_{u_i^\perp}(M)$, we obtain
\[
    I(K+L)^{\frac1{n+2}}\ge I(K)^{\frac1{n+2}}+I(L)^{\frac1{n+2}}.
\]
This proves the inequality.

For equality, equality must hold in Step~1, hence in the Brunn-Minkowski inequality in $\R^{n+2}$ for the associated lifted bodies. Therefore the relevant one-sided lifted bodies are homothetic, which forces the original convex bodies to be homothetic. Conversely, homothetic bodies give equality by homogeneity.
\end{proof}

We now return to the endpoint $q=n+2$. By \eqref{polar-dual-q},
\[
    \tV_{n+2}(K)=\frac{n+2}{n}\int_K |x|^2\,\dd x.
\]
If $K$ is origin-symmetric, then $g(K)=0$, so
\[
    I(K)=\int_K |x|^2\,\dd x=\frac{n}{n+2}\tV_{n+2}(K).
\]
Therefore Hadwiger's inequality immediately gives Theorem~\ref{thm-endpoint}.

\begin{proof}[Proof of Theorem~\ref{thm-endpoint}]
Let $K,L\subset\R^n$ be origin-symmetric convex bodies. Then $g(K)=g(L)=g(K+L)=0$, and hence
\[
    I(M)=\frac{n}{n+2}\tV_{n+2}(M),\qquad M=K,L,K+L.
\]
Applying Theorem~\ref{thm-Hadwiger}, we get
\[
    \tV_{n+2}(K+L)^{\frac1{n+2}}
    \ge \tV_{n+2}(K)^{\frac1{n+2}}+\tV_{n+2}(L)^{\frac1{n+2}}.
\]
If equality holds, then $K$ and $L$ are homothetic by Theorem~\ref{thm-Hadwiger}. Since both are origin-symmetric, the translation part vanishes, so they are dilates of each other.
\end{proof}

\section{The Reilly approach}
\label{sec-unconditional-prep}

Let $\alpha=q-n \in (-n,1]$. We denote
\[
  \tilde{\mu}_\alpha(K):=\int_K |x|^\alpha\,\dd x .
\]
Then $\tilde{\mu}_\alpha$ is $(n+\alpha)$-homogeneous and $\tV_q(K)=\frac qn\tilde{\mu}_\alpha(K)$. 
Hence Theorem~\ref{thm-unconditional} is equivalent to
\begin{equation}\label{tilde-mu-alpha-BM}
  \tilde{\mu}_\alpha((1-\lambda) K+\lambda L)
  \geq
  \tilde{\mu}_\alpha(K)^{1-\lambda} \tilde{\mu}_\alpha(L)^{\lambda},\quad -n<\alpha\le 1,
\end{equation}
for all $\lambda\in[0,1]$ and for all unconditional convex bodies $K,L\subset\bR^{n}$.

Define the potential and the weighted measure
\[
  W_\alpha(x)=-\alpha\log|x|,
  \qquad
  \dd\tilde{\mu}_\alpha(x)=e^{-W_\alpha}\dd x=|x|^\alpha\dd x .
\]
On $\R^n\setminus\{0\}$, the associated elliptic operator is
\[
  L_\alpha u
  =\Delta u-\langle\nabla W_\alpha,\nabla u\rangle
  =\Delta u+\alpha\,\frac{x\cdot\nabla u}{|x|^2}
  =|x|^{-\alpha}\divg\!\bigl(|x|^\alpha\nabla u\bigr).
\]

For a bounded Lipschitz domain $K\subset\R^n$ containing the origin, set
\[
  H^1(K,\tilde{\mu}_\alpha)
  =
  \bigl\{u\in L^2(K,\tilde{\mu}_\alpha):
  \nabla u\in L^2(K,\tilde{\mu}_\alpha)\bigr\},
\]
where derivatives are taken in the distributional sense. The weight
$|x|^\alpha$ belongs to the Muckenhoupt class $A_2(\R^n)$ for
$|\alpha|<n$. Thus the space $H^1(K,\tilde{\mu}_\alpha)$ is a Hilbert
space and supports the weighted Poincar\'e theory. Since the
singularity is away from $\partial K$, the ordinary trace theorem on a
collar neighbourhood of $\partial K$ gives the boundary trace estimates
used below. We refer to \cite{FKS82} and \cite{HKM06} for the general
theory of degenerate elliptic equations with $A_2$ weights.

\subsection{The singular Neumann problem}
\label{sec-singular-neumann}

\begin{lem}\label{lem-neumann}
  Let $-n<\alpha<n$, and let $K\in\mathcal{K}_{+}^n$. Let
  $\varphi\in C^\infty(\partial K)$ and set
  \[
    \dd\tilde{\mu}_\alpha=|x|^\alpha\,\dd x,
    \qquad
    \dd\tilde{\mu}_{\alpha,\partial K}
    =
    |x|^\alpha\,\dd H^{n-1},
  \]
  and
  \[
    C_\varphi
    =
    \frac{
      \int_{\partial K}\varphi\,\dd\tilde{\mu}_{\alpha,\partial K}}
    {      \tilde{\mu}_\alpha(K)}.
  \]
  Then the normalized weighted Neumann problem
  \[
    L_\alpha u=C_\varphi\quad\text{in }K,
    \qquad
    u_\nu=\varphi\quad\text{on }\partial K,
    \qquad
    \int_K u\,\dd\tilde{\mu}_\alpha=0
  \]
  has a unique weak solution
  $u\in H^1_\diamond(K,\tilde{\mu}_\alpha)$, where
  \[
    H^1_\diamond(K,\tilde{\mu}_\alpha)
    =
    \left\{
      v\in H^1(K,\tilde{\mu}_\alpha):
      \int_K v\,\dd\tilde{\mu}_\alpha=0
    \right\}.
  \]
  Equivalently, $u$ is the unique element of
  $H^1_\diamond(K,\tilde{\mu}_\alpha)$ satisfying
  \begin{equation}\label{eq-neumann-weak}
    \int_K\langle\nabla u,\nabla\eta\rangle\,
    \dd\tilde{\mu}_\alpha
    =
    -C_\varphi\int_K\eta\,\dd\tilde{\mu}_\alpha
    +
    \int_{\partial K}\varphi\eta\,
    \dd\tilde{\mu}_{\alpha,\partial K}
  \end{equation}
  for every $\eta\in H^1(K,\tilde{\mu}_\alpha)$.

    Moreover,
  \[
    u\in C^\infty(\operatorname{int}K\setminus\{0\}),
    \qquad
    L_\alpha u=C_\varphi\quad\text{in }K\setminus\{0\},
    \qquad
    u_\nu=\varphi\quad\text{on }\partial K,
  \]
  and, for every $0<\varepsilon<\operatorname{dist}(0,\partial K)$,
  \[
    u\in W^{2,2}\bigl(K\setminus\overline{B_\varepsilon}\bigr).
  \]
  If in addition $K$ and $\varphi$ are unconditional, then $u$ is
  unconditional.
\end{lem}

\begin{proof}
  Write $w(x)=|x|^\alpha$, so that
  $\dd\tilde{\mu}_\alpha=w\,\dd x$. Since $0\in\operatorname{Int}K$,
  the weight is smooth and bounded above and below by positive
  constants in a collar neighbourhood of $\partial K$. Hence the
  ordinary trace theorem in this collar gives
  \begin{equation}\label{eq-trace-estimate-tilde}
    \|\eta\|_{L^2(\partial K,\tilde{\mu}_{\alpha,\partial K})}
    \le C\|\eta\|_{H^1(K,\tilde{\mu}_\alpha)}
    \qquad
    \text{for all }\eta\in H^1(K,\tilde{\mu}_\alpha).
  \end{equation}
  Therefore,
  \[
    F(\eta):=
    -C_\varphi\int_K\eta\,\dd\tilde{\mu}_\alpha
    +
    \int_{\partial K}\varphi\eta\,
    \dd\tilde{\mu}_{\alpha,\partial K}
  \]
  defines a bounded linear functional on
  $H^1(K,\tilde{\mu}_\alpha)$. It follows from the definition of $C_\varphi$ that $F(1)=0$.
  Since $|x|^\alpha\in A_2(\mathbb R^n)$ for $|\alpha|<n$, the weighted
  Poincar\'e inequality holds on the bounded connected Lipschitz domain
  $K$:
  \begin{equation}\label{eq-wpoincare-tilde}
    \int_K |v-v_{\tilde{\mu}_\alpha}|^2
    \,\dd\tilde{\mu}_\alpha
    \le
    C_{K,\alpha}
    \int_K |\nabla v|^2\,\dd\tilde{\mu}_\alpha ,
  \end{equation}
  where $ v_{\tilde{\mu}_\alpha}
    =
    \frac1{\tilde{\mu}_\alpha(K)}
    \int_K v\,\dd\tilde{\mu}_\alpha$.
  Thus, on $H^1_\diamond(K,\tilde{\mu}_\alpha)$, the bilinear form
  \[
    B(v,\eta)
    =\int_K\langle\nabla v,\nabla\eta\rangle
    \,\dd\tilde{\mu}_\alpha
  \]
  is bounded and coercive. The Lax-Milgram theorem gives a unique
  $u\in H^1_\diamond(K,\tilde{\mu}_\alpha)$ such that
  \[
    B(u,\eta)=F(\eta)
    \qquad
    \text{for all }\eta\in H^1_\diamond(K,\tilde{\mu}_\alpha).
  \]
    For any $\eta=\eta_0+\eta_{\tilde{\mu}_\alpha}\in H^1(K,\tilde{\mu}_\alpha)$, where $\eta_{\tilde{\mu}_\alpha}$ is constant and $\eta_0\in H^1_\diamond(K,\tilde{\mu}_\alpha)$, by $F(1)=0$, we get
  \[
    B(u,\eta)=B(u,\eta_0)=F(\eta_0)=F(\eta),
  \]
  which proves \eqref{eq-neumann-weak} for every
  $\eta\in H^1(K,\tilde{\mu}_\alpha)$.

  Taking test functions in $C_c^\infty(K\setminus\{0\})$ gives
  \[
    |x|^{-\alpha}\divg(|x|^\alpha\nabla u)=C_\varphi
  \]
  in the sense of distributions on $K\setminus\{0\}$, that is,
  $L_\alpha u=C_\varphi$. The boundary term in
  \eqref{eq-neumann-weak} gives the weak conormal condition
  $|x|^\alpha u_\nu=|x|^\alpha\varphi$ on $\partial K$. Since
  $|x|^\alpha$ is smooth and strictly positive near $\partial K$, this
  is the usual Neumann condition $u_\nu=\varphi$.
  On every compact subset of $\operatorname{int}K\setminus\{0\}$, the operator
  $L_\alpha$ is uniformly elliptic with smooth coefficients. Standard
  interior regularity implies \(u\in C^\infty(\operatorname{int}K\setminus\{0\})\),
  and the equation and boundary condition hold classically away from
  the origin. The claimed $W^{2,2}$ regularity up to the boundary
  follows from standard Neumann estimates on the $C^2$ domain
  $K\setminus\overline{B_\varepsilon}$; see, e.g., \cite{Lie13}.
  Uniqueness follows by testing the difference of two solutions.

  If $K$ and $\varphi$ are unconditional, then for every coordinate
  reflection $\sigma$ the function $u_\sigma(x)=u(\sigma x)$ has the
  same weighted mean, satisfies the same weak problem, and belongs to
  $H^1_\diamond(K,\tilde{\mu}_\alpha)$. By uniqueness,
  $u_\sigma=u$. Thus $u$ is unconditional.
\end{proof}

\begin{rmk}
  For $K\in\mathcal{K}_{+}^n$, the conclusion of Lemma~\ref{lem-neumann} remains true by first smoothing $\varphi$ in the $C^1$ topology and then passing to the limit. 
  Indeed, the trace estimate \eqref{eq-trace-estimate-tilde},
  the weighted Poincar\'e inequality \eqref{eq-wpoincare-tilde}, and
  the Lax-Milgram bound are all stable under uniform $C^1$
  approximation of $\varphi$; the resulting solutions converge in
  $H^1_\diamond(K,\tilde{\mu}_\alpha)$ to the desired weak solution.
\end{rmk}

Because the density $|x|^\alpha$ may be singular at the origin,
the solution $u$ is not \emph{a priori} smooth there in the negative
parameter case. The following expansion provides the precise asymptotic
behaviour needed to justify the singular Reilly formula. Let
$q=n+\alpha>0$; note that $q\in(0,n+1]$ in the full range.

\begin{lem}\label{lem-expansion} 
  Let $n\ge2$, $-n<\alpha<n$, and let $u$ be the unconditional Neumann
  solution from Lemma~\ref{lem-neumann} on a ball $B_R\subset K$.
  Set $q=n+\alpha$. Then there exist coefficients $a_{m,\ell}$ such that for every
  $R_0<R$,
  \begin{equation}\label{eq-expansion}
    u(r,\theta)
    =a_0+\frac{C_\varphi}{2q}r^2
     +\sum_{\substack{m\ge2\\m\ \text{even}}}
       \sum_{\ell=1}^{d_m^{\rm unc}} a_{m,\ell}\,
       r^{\beta_m}Y_{m,\ell}^{\rm unc}(\theta),
  \end{equation}
  where $\{Y_{m,\ell}^{\rm unc}\}$ is an orthonormal basis of the
  unconditional spherical harmonics of degree $m$, and
  \begin{equation}\label{eq-beta-m}
    \beta_m
    =\frac{-(q-2)+\sqrt{(q-2)^2+4m(m+n-2)}}{2}.
  \end{equation}
  The series converges in $H^1(B_{R_0},\mu_\alpha)$ and, together with
  all its derivatives, locally uniformly on $B_{R_0}\setminus\{0\}$.
  Consequently, as $r\downarrow0$,
  \begin{equation}\label{eq-asymptotics}
    u=a_0+O(r^{\beta_2}+r^2),\qquad
    |\nabla u|=O(r^{\beta_2-1}+r),\qquad
    |D^2u|=O(r^{\beta_2-2}+1),
  \end{equation}
  where $\beta_2$ is the exponent in \eqref{eq-beta-m} for $m=2$.
  Moreover,
  \begin{equation}\label{eq-hessian-l2}
    D^2u\in L^2(B_{R/2},\mu_\alpha).
  \end{equation}
\end{lem}

\begin{proof}
  Let $v=u-\frac{C_\varphi}{2q}|x|^2$. Since $L_\alpha(|x|^2)=2q$, the function $v$ is unconditional,
  belongs to $H^1(B_R,\tilde{\mu}_\alpha)$, and satisfies $L_\alpha v=0$
  weakly in $B_R$.  Standard interior elliptic regularity
  \cite[Chapters~6,8]{GT01} gives $v\in C^\infty(B_R\setminus\{0\})$.

  In polar coordinates,
  \[
    L_\alpha
    =\partial_{rr}+\frac{q-1}{r}\partial_r
     +\frac1{r^2}\Delta_{\mathbb{S}^{n-1}} .
  \]
  Let $\cH_m$ denote the space of spherical harmonics of degree $m$
  on $\mathbb{S}^{n-1}$, with $-\Delta_{\mathbb{S}^{n-1}}Y_m=m(m+n-2)Y_m$.  The
  orthogonal decomposition $L^2(\mathbb{S}^{n-1})=\bigoplus_{m=0}^\infty\cH_m$
  is standard; see \cite[Appendix]{Sch14} or
  \cite[Section~3]{Gro96}.  Let $v_m(r)\in\cH_m$ be the projection
  of $v(r,\cdot)$.  Projecting the equation $L_\alpha v=0$ onto
  $\cH_m$ yields the ODE
  \begin{equation}\label{eq-ode}
    v_m''+\frac{q-1}{r}v_m'
    -\frac{m(m+n-2)}{r^2}v_m=0,\qquad 0<r<R .
  \end{equation}

  Looking for solutions of the form $r^\gamma$ gives the indicial
  equation $\gamma(\gamma+q-2)=m(m+n-2)$, whose roots are
  precisely $\beta_m^+=\beta_m$ (as in \eqref{eq-beta-m}) and
  $\beta_m^-=-(q-2)-\beta_m$.  Hence
  \[
    v_m(r)=r^{\beta_m}A_m+r^{\beta_m^-}B_m,
  \]
  with the usual modification when $m=0$ (the two roots are $0$ and
  $2-q$, giving $v_0(r)=A_0+B_0 r^{2-q}$, with the second solution
  $\log r$ if $q=2$).

  \emph{Exclusion of the singular branch.}
  For $m\ge1$ and a non‑zero $Y_m\in\cH_m$,
  \[
    \int_{B_\varepsilon}|\nabla(r^\gamma Y_m)|^2\,\dd\tilde{\mu}_\alpha
    \asymp\int_0^\varepsilon r^{2\gamma+q-3}\,\dd r ,
  \]
  where $A\asymp B$ if $cB\le A\le CB$ for some universal constants $c,C>0$. Substituting $\gamma=\beta_m^-$ and using
  $\beta_m^-=-(q-2)-\beta_m$ gives the exponent
  $-(q-1)-2\beta_m\le-(q-1)<-1$ (since $q>0$).  Hence the integral diverges as $\varepsilon\to0$,
  contradicting $v\in H^1(B_R,\tilde{\mu}_\alpha)$.  Thus $B_m=0$ for all $m\ge1$.
  For $m=0$, the singular solution $r^{2-q}$ (or $\log r$ when $q=2$) may have finite weighted
  $H^1$ energy when $0<q<2$.  It is nevertheless excluded because it is not a weak
  solution across the origin: indeed,
  \[
    r^{q-1}\partial_r(r^{2-q})=2-q,\qquad
    r^{q-1}\partial_r(\log r)=1\quad(q=2),
  \]
  so a nonzero coefficient produces a nonzero limiting flux through small spheres,
  i.e.\ a Dirac mass at the origin in
  $\operatorname{div}(|x|^\alpha\nabla u)$.  The weak Neumann equation has only the
  prescribed constant right-hand side $C_\varphi$ and no Dirac mass.
  Therefore the coefficient of this mode is zero.

  Therefore $v_m(r)=r^{\beta_m}g_m$ for every $m$, where
  $g_m=v_m(R_0,\cdot)$ for some fixed $0<R_0<R$ and we set
  $\beta_0=0$ (the constant solution).

  \emph{Unconditional symmetry.}
  Because $v$ is unconditional, each $g_m$ belongs to the subspace
  $\cH_m^{\rm unc}$ of spherical harmonics invariant under all
  coordinate reflections.  A homogeneous harmonic polynomial in
  $\cH_m^{\rm unc}$ contains only monomials with even exponent in
  each variable; therefore $m$ must be even.  Thus
  $\cH_m^{\rm unc}=\{0\}$ for odd $m$, and only even $m$ appear in
  the expansion.

  \emph{Convergence and asymptotics.}
  The trace $v(R_0,\cdot)$ is smooth on $\mathbb{S}^{n-1}$.  Standard
  elliptic estimates on the sphere \cite[Section~3]{Gro96},
  \cite[Chapter~III]{Cha84} imply that its spherical harmonic
  coefficients decay faster than any negative power of~$m$:
  $\|g_m\|_{C^k(\mathbb{S}^{n-1})}\le C_{k,N}\,(1+m)^{-N}$ for every
  $k,N\ge0$.  Since $\beta_m\asymp m$, the exponential factor
  $(r/R_0)^{\beta_m}\le (r/R_0)^{\beta_2}\,2^{-(\beta_m-\beta_2)}$
  (for $r\le R_0/2$) dominates all polynomial factors.  The series
  therefore converges in $H^1(B_{R_0},\tilde{\mu}_\alpha)$ and may be
  differentiated term by term on compact subsets of
  $B_{R_0}\setminus\{0\}$.  This yields the asymptotics
  \eqref{eq-asymptotics}.

  Finally, $|D^2u|^2=O(r^{2\beta_2-4}+1)$.  The radial integral
  $\int_0^\delta r^{2\beta_2+q-5}\,\dd r$ converges because
  \[
    2\beta_2+q-4
    =\sqrt{(q-2)^2+8n}-2>0.
  \]
  Hence $D^2u\in L^2(B_{R/2},\tilde{\mu}_\alpha)$, proving
  \eqref{eq-hessian-l2}.
\end{proof}

\subsection{The singular weighted Reilly formula}\
\label{sec-singular-reilly}

We recall the generalized weighted Reilly formula from
Kolesnikov-Milman~\cite[Theorem~1.1]{KM18}. For a smooth bounded
domain $\Omega$, a smooth potential $W$, and
$u\in C^\infty(\overline\Omega)$,
\begin{equation}\label{eq-reilly-smooth}
  \begin{aligned}
    \int_\Omega(L_W u)^2 e^{-W}\dd x
    =&\int_\Omega\bigl(\|D^2u\|_{\HS}^2
       +D^2W(\nabla u,\nabla u)\bigr)e^{-W}\dd x
       +\int_{\partial\Omega}H_W u_\nu^2 e^{-W}\dd H^{n-1}\\
    &+\int_{\partial\Omega}\mathrm{II}(\nabla_{\partial\Omega}u,
       \nabla_{\partial\Omega}u)\,e^{-W}\dd H^{n-1}
       -2\int_{\partial\Omega}
       \langle\nabla_{\partial\Omega}u_\nu,
       \nabla_{\partial\Omega}u\rangle\,e^{-W}\dd H^{n-1},
  \end{aligned}
\end{equation}
where $H_W=H-\langle\nabla W,\nu\rangle$ and $\mathrm{II}(X,Y)=\langle D_X\nu,Y\rangle$.

Our potential $W_\alpha=-\alpha\log|x|$ is singular at the origin. We
apply \eqref{eq-reilly-smooth} on
$K_\varepsilon=K\setminus\overline B_\varepsilon$ and let
$\varepsilon\downarrow0$.

\begin{lem}
  \label{lem-inner-boundary}
  Let $K$ and $u$ be as in Lemma~\ref{lem-neumann}. For
  $\varepsilon>0$ small, let $\Sigma_\varepsilon=\partial B_\varepsilon$
  be the inner boundary of $K_\varepsilon$, with outward unit normal
  $\nu_\varepsilon=-e_r$ relative to $K_\varepsilon$. Then
  \begin{equation}\label{eq-inner-vanishing}
    \begin{aligned}
      \mathcal B_\varepsilon(u)
      :=&\int_{\Sigma_\varepsilon}
         H_{W_\alpha,\varepsilon}u_{\nu_\varepsilon}^2
         \,\dd\tilde{\mu}_{\alpha,\Sigma_\varepsilon}
        +\int_{\Sigma_\varepsilon}
         \mathrm{II}_\varepsilon(\nabla_{\Sigma_\varepsilon}u,
         \nabla_{\Sigma_\varepsilon}u)
         \,\dd\tilde{\mu}_{\alpha,\Sigma_\varepsilon}\\
        &-2\int_{\Sigma_\varepsilon}
         \langle\nabla_{\Sigma_\varepsilon}u_{\nu_\varepsilon},
         \nabla_{\Sigma_\varepsilon}u\rangle
         \,\dd\tilde{\mu}_{\alpha,\Sigma_\varepsilon}
    \end{aligned}
  \end{equation}
  tends to $0$ as $\varepsilon\downarrow0$.
\end{lem}

\begin{proof}
  On $\Sigma_\varepsilon$ we have $\nu_\varepsilon=-e_r$.  For
  $X,Y\in T\Sigma_\varepsilon$,
  $D_X\nu_\varepsilon=-\frac1\varepsilon X$, whence
  $\mathrm{II}_\varepsilon(X,Y)=-\frac1\varepsilon\langle X,Y\rangle$
  and $H_\varepsilon=-\frac{n-1}{\varepsilon}$.  Moreover
  $\nabla W_\alpha=-\frac\alpha\varepsilon e_r$ on
  $\Sigma_\varepsilon$, so
  $\langle\nabla W_\alpha,\nu_\varepsilon\rangle
   =\frac\alpha\varepsilon$ and
  \[
    H_{W_\alpha,\varepsilon}
    =H_\varepsilon-\langle\nabla W_\alpha,\nu_\varepsilon\rangle
    =-\frac{q-1}{\varepsilon}
  \]
  since $q=n+\alpha$.
  The weighted surface measure is
  $\dd\tilde{\mu}_{\alpha,\Sigma_\varepsilon}
   =\varepsilon^{q-1}\dd\theta$.

  The asymptotic estimates \eqref{eq-asymptotics} give, for
  $x\in\Sigma_\varepsilon$,
  \[
    |u_{\nu_\varepsilon}|+|\nabla_{\Sigma_\varepsilon}u|
    \le C(\varepsilon^{\beta_2-1}+\varepsilon),\qquad
    |\nabla_{\Sigma_\varepsilon}u_{\nu_\varepsilon}|
    \le C(\varepsilon^{\beta_2-2}+1).
  \]
  Consequently,
  \[
    \begin{aligned}
      \bigl|\mathcal B_\varepsilon(u)\bigr|
      &\le C\varepsilon^{-1}
         \bigl(\varepsilon^{\beta_2-1}+\varepsilon\bigr)^2
         \varepsilon^{q-1}
        +C\bigl(\varepsilon^{\beta_2-2}+1\bigr)
          \bigl(\varepsilon^{\beta_2-1}+\varepsilon\bigr)
          \varepsilon^{q-1}\\[2pt]
      &\le C\bigl(
         \varepsilon^{2\beta_2+q-4}
        +\varepsilon^{\beta_2+q-2}
        +\varepsilon^{q}\bigr).
    \end{aligned}
  \]
  The dominant exponent is $2\beta_2+q-4$, which is strictly
  positive because
  $2\beta_2+q-4=\sqrt{(q-2)^2+8n}-2>0$ for all
  $n\ge2$ and $q>0$.  Hence $\mathcal B_\varepsilon(u)\to0$.
\end{proof}

We also need the Hessian of the singular potential:
\[
  D^2W_\alpha(x)
  =
  -\frac\alpha{|x|^2}\bigl(\Id-2e_r\otimes e_r\bigr),
  \qquad
  e_r=\frac x{|x|}.
\]
For $\nabla_Tu=\nabla u-(\partial_ru)e_r$, this gives
\begin{equation}\label{eq-hessian-split}
  D^2W_\alpha(\nabla u,\nabla u)
  =
  -\alpha\frac{|\nabla_Tu|^2}{|x|^2}
  +
  \alpha\frac{|\partial_ru|^2}{|x|^2}.
\end{equation}

\begin{prop}[Singular weighted Reilly formula]
  \label{prop-singular-reilly}
  Under the assumptions of Lemma~\ref{lem-neumann},
  \begin{equation}\label{eq-reilly-final}
    \begin{aligned}
      C_\varphi^2\tilde{\mu}_\alpha(K)
      =&\int_K\bigl(\|D^2u\|_{\HS}^2
         +D^2W_\alpha(\nabla u,\nabla u)\bigr)
         \dd\tilde{\mu}_\alpha
       +\int_{\partial K}H_{W_\alpha}u_\nu^2
         \,\dd\tilde{\mu}_{\alpha,\partial K}\\
       &+\int_{\partial K}\mathrm{II}(\nabla_{\partial K}u,
         \nabla_{\partial K}u)
         \,\dd\tilde{\mu}_{\alpha,\partial K}
       -2\int_{\partial K}
         \langle\nabla_{\partial K}u_\nu,
         \nabla_{\partial K}u\rangle
         \,\dd\tilde{\mu}_{\alpha,\partial K}.
    \end{aligned}
  \end{equation}
\end{prop}

\begin{proof}
  If $K$ is $C^\infty$, apply \eqref{eq-reilly-smooth} directly to
  $K_\varepsilon=K\setminus\overline{B_\varepsilon}$ with $W=W_\alpha$.
  For $K\in\mathcal{K}_+^n$ (i.e.\ merely $C^2$), the formula follows by
  first smoothing $K$ and approximating $u$ in
  $W^{2,2}(K\setminus\overline{B_\varepsilon})$, then passing to the
  limit; all boundary traces below are well-defined by the
  $W^{1,2}$ trace theorem on the $C^2$ boundary. Apply \eqref{eq-reilly-smooth} to $K_\varepsilon$ with
  $W=W_\alpha$. The left-hand side equals \(C_\varphi^2\tilde{\mu}_\alpha(K_\varepsilon)
    \to
    C_\varphi^2\tilde{\mu}_\alpha(K).\)
  The inner boundary contribution $\mathcal B_\varepsilon(u)$ tends to zero by Lemma~\ref{lem-inner-boundary}. The outer boundary terms are
  independent of $\varepsilon$. By Lemma~\ref{lem-expansion}, $(\|D^2u\|^2+|D^2W_\alpha(\nabla u,\nabla u)|)|x|^\alpha$ is integrable near the origin, so the integral over
  $B_\varepsilon$ vanishes as $\varepsilon\downarrow0$. Passing to the
  limit yields \eqref{eq-reilly-final}.
\end{proof}

\section{Hardy estimates and proof of Theorem~\ref{thm-unconditional}}
\label{sec-unconditional-proof}

In this section we prove the Hardy estimates needed to control the
unfavorable terms in the singular Reilly formula, and then
derive Theorem~\ref{thm-unconditional}. 

\subsection{Hardy-type estimates on unconditional convex bodies}\
\label{sec-hardy-slices}

The following inequality is the analytic core of the coordinate‑slice
argument for positive $\alpha$. The offset parameter $c$ is essential for obtaining the
optimal constant.

\begin{lem}\label{lem-hardy-offset} 
  Let $0<\alpha<3$, $c>0$, $R>0$, and $f\in C^1([0,R])$ with
  $f(0)=0$.  Then
  \begin{equation}\label{eq-hardy-offset}
    \int_0^R|f'(t)|^2(t^2+c^2)^{\alpha/2}\,\dd t
    \ge(3-\alpha)c^2
    \int_0^R f(t)^2(t^2+c^2)^{(\alpha-4)/2}\,\dd t.
  \end{equation}
\end{lem}

\begin{proof}
  By the change of variables $t=cs$ and $F(s)=f(cs)$, both sides of
  \eqref{eq-hardy-offset} are multiplied by the same factor
  $c^{\alpha-1}$.  It suffices to prove the case $c=1$. Let $t=\sinh s$, $S=\operatorname{arcsinh}R$, and $g(s)=f(\sinh s)$. Then the desired inequality is equivalent to
  \begin{equation}\label{eq-hardy-offset-2}
       \int_0^S |g'(s)|^2\cosh^{\alpha-1}s\,\dd s
    \ge
    (3-\alpha)
    \int_0^S g(s)^2\cosh^{\alpha-3}s\,\dd s .
  \end{equation}

  Set $w(s)=\cosh^{\alpha-1}s$ and $\phi(s)=\tanh s>0$. We have
    \[-\frac{(w\phi')'}{\phi}
    =(3-\alpha)\cosh^{\alpha-3}s.\]
  For $\delta\in(0,S)$, the non-negativity of the square gives
  \[
    0
    \le
    \int_\delta^S
    w\left|g'-\frac{\phi'}{\phi}g\right|^2\,\dd s .
  \]
  Expanding and integrating the mixed term by parts, we obtain
  \[
    \int_\delta^S w|g'|^2\,\dd s
    -
    \int_\delta^S
    \left(-\frac{(w\phi')'}{\phi}\right)g^2\,\dd s
    \ge
    \left[
      w\frac{\phi'}{\phi}g^2
    \right]_{\delta}^{S}.
  \]
  The boundary contribution at $S$ is non-negative. At the lower
  endpoint, since $g(0)=0$ and $g\in C^1([0,S])$, we have
  $g(s)=O(s)$, while $ \phi(s)\sim s,\,\phi'(s)\to1$ and $w(s)\to1$ as $s\downarrow0$. Hence
  \[
    w(\delta)\frac{\phi'(\delta)}{\phi(\delta)}g(\delta)^2
    =O(\delta)\to0 .
  \]
  Letting $\delta\downarrow0$ yields \eqref{eq-hardy-offset-2}. We complete the proof.
\end{proof}

\begin{rmk}
  The constant $(3-\alpha)$ in Lemma~\ref{lem-hardy-offset} is sharp.  Replacing the offset weight by
  the classical one $t^\alpha$ (as in the standard Hardy inequality)
  would yield only the constant $(1-\alpha)^2/4$, which is too weak
  to reach $\alpha=1$.
\end{rmk}

On the other hand, we establish the estimate for negative $\alpha$. Set
$a=-\alpha>0$.

\begin{lem}\label{lem-hardy-neg}
Let $a>0$, $c>0$, $R>0$, and $f\in C^1([0,R])$ with $f(0)=0$. Then
\begin{equation}\label{eq-hardy-neg}
 \int_0^R |f'(t)|^2(t^2+c^2)^{-a/2}\,\dd t
 \ge
 a\int_0^R f(t)^2(t^2+c^2)^{-(a+2)/2}\,\dd t.
\end{equation}
\end{lem}

\begin{proof}
By the change of variables $t=cs$ and $F(s)=f(cs)$, both sides of
\eqref{eq-hardy-neg} are multiplied by the same factor
$c^{-a-1}$.  It suffices to prove the case $c=1$. Let $t=\tan s$, $S=\arctan R$, and $g(s)=f(\tan s)$. Then the desired inequality is equivalent to
\begin{equation}\label{eq-hardy-neg-2}
     \int_0^S |g'(s)|^2\cos^{a+2}s\,\dd s
  \ge
     a \int_0^S g(s)^2\cos^a s\,\dd s .
\end{equation}

Set $w(s)=\cos^{a+2}s$ and $\phi(s)=\tan s>0$. We have
\[
 -\frac{(w\phi')'}{\phi}
 = a\cos^a s.
\]
For $\delta\in(0,S)$, the non-negativity of the square gives
\[
 0
 \le
 \int_\delta^S
 w\left|g'-\frac{\phi'}{\phi}g\right|^2\,\dd s .
\]
Expanding and integrating the mixed term by parts, we obtain
\[
 \int_\delta^S w|g'|^2\,\dd s
 -
 \int_\delta^S
 \left(-\frac{(w\phi')'}{\phi}\right)g^2\,\dd s
 \ge
 \left[
   w\frac{\phi'}{\phi}g^2
 \right]_{\delta}^{S}.
\]
The boundary contribution at $S$ is non-negative. At the lower
endpoint, since $g(0)=0$ and $g\in C^1([0,S])$, we have
$g(s)=O(s)$, while $ \phi(s)\sim s,\,\phi'(s)\to1$ and $w(s)\to1$ as $s\downarrow0$. Hence
\[
 w(\delta)\frac{\phi'(\delta)}{\phi(\delta)}g(\delta)^2
 =O(\delta)\to0 .
\]
Letting $\delta\downarrow0$ yields \eqref{eq-hardy-neg-2}. This completes the proof.
\end{proof}

\begin{lem}\label{lem-slice} 
Let $K\subset\R^n$ be an unconditional convex body, and let $u\in C^\infty(\operatorname{int}K\setminus\{0\})$ be unconditional with $D^2u\in L^2(K,|x|^\alpha\,dx)$, where either $0<\alpha<3$, or $-n<\alpha<0$. Then the following hold.
\begin{enumerate}[(i)]
\item  If $0<\alpha<3$, then for every $i=1,\dots,n$,
\begin{equation}\label{eq-slice-positive}
 \int_K u_{ii}^2|x|^\alpha\,dx
 \ge
 (3-\alpha)\int_K (|x|^2-x_i^2)\,u_i^2\,|x|^{\alpha-4}\,dx .
\end{equation}

\item If $-n<\alpha<0$, then for every $i=1,\dots,n$, 
\begin{equation}\label{eq-slice-negative}
 \int_K u_{ii}^2|x|^{\alpha}\,dx
 \ge -\alpha\int_K u_i^2 |x|^{\alpha-2}\,dx .
\end{equation}
\end{enumerate}
\end{lem}

\begin{proof}
Fix $i$ and write $x=te_i+x'$ with $t=x_i\in\R$,
$x'\in e_i^\perp$. Since $K$ is unconditional, the coordinate
slice in the $x_i$-direction is a symmetric interval
$[-R_i(x'),R_i(x')]$.  Since $u$ is unconditional, $u_i$ is odd
in the $i$-th variable, hence $u_i(0,x')=0$.

For $x'\neq0$ with non-empty slice, define
$f(t)=u_i(te_i+x')$ on $[0,R_i(x')]$. The exceptional slice
$x'=0$ has zero measure in $e_i^\perp$ and is irrelevant for integration.

\begin{enumerate}[(i)]
\item 
Lemma~\ref{lem-hardy-offset} with $c=|x'|$ gives
\[
 \int_0^{R_i}
 u_{ii}(te_i+x')^2(t^2+|x'|^2)^{\alpha/2}\,\dd t
 \ge(3-\alpha)|x'|^2
 \int_0^{R_i}
 u_i(te_i+x')^2(t^2+|x'|^2)^{(\alpha-4)/2}\,\dd t.
\]
Integrating over $x'\in e_i^\perp$ and applying Tonelli's theorem gives
\eqref{eq-slice-positive}.

\item 
Lemma~\ref{lem-hardy-neg} with $c=|x'|$ and $a=-\alpha$ gives
\[
 \int_0^{R_i}
 u_{ii}(te_i+x')^2(t^2+|x'|^2)^{\alpha/2}\,\dd t
 \ge
 -\alpha
 \int_0^{R_i} u_i(te_i+x')^2(t^2+|x'|^2)^{(\alpha-2)/2}\,\dd t,
\]
Integrating over $x'\in e_i^\perp$ gives \eqref{eq-slice-negative}.

\end{enumerate}
\end{proof}

For $x\neq0$, decompose the gradient as
\[\nabla u=\partial_r u\,e_r+\nabla_T u,\]
where $e_r=\frac{x}{|x|}$ and $\partial_r u=\langle\nabla u,e_r\rangle$.

\begin{prop}\label{prop-tangential}
  Under the assumptions of Lemma~\ref{lem-slice}, the following estimates hold.
  \begin{enumerate}[(i)]
  \item If $0<\alpha<3$, then
  \begin{equation}\label{eq-tangential-pos}
    \int_K|\nabla_T u|^2|x|^{\alpha-2}\dd x
    \le \frac{2}{3-\alpha} \sum_{i=1}^n\int_K u_{ii}^2|x|^\alpha\dd x.
  \end{equation}

  \item If $-n<\alpha<0$, then
  \begin{equation}\label{eq-nabla-neg}
    \int_K|\nabla u|^2|x|^{\alpha-2}\dd x
    \le \frac{1}{-\alpha} \sum_{i=1}^n\int_K u_{ii}^2|x|^\alpha\dd x.
  \end{equation}
   \end{enumerate}
\end{prop}

\begin{proof}
  The elementary identity
  \[
    |x|^2|\nabla_T u|^2
    =|x|^2|\nabla u|^2-(x\cdot\nabla u)^2
    =\sum_{1\le i<j\le n}(x_i u_j-x_j u_i)^2
  \]
  together with $(a-b)^2\le2a^2+2b^2$ yields
  \[
    |x|^2|\nabla_T u|^2
    \le2\sum_{i<j}\bigl(x_i^2u_j^2+x_j^2u_i^2\bigr)
    =2\sum_{i=1}^n\bigl(|x|^2-x_i^2\bigr)u_i^2.
  \]
  Multiplying by $|x|^{\alpha-4}$ and integrating gives
  \begin{equation}\label{est-grad-T}
       \int_K|\nabla_T u|^2|x|^{\alpha-2}\dd x
    \le2\sum_{i=1}^n
    \int_K\bigl(|x|^2-x_i^2\bigr)u_i^2|x|^{\alpha-4}\dd x.
  \end{equation}

  \begin{enumerate}[(i)]
  \item If $0<\alpha<3$, it follows from \eqref{est-grad-T} and \eqref{eq-slice-positive} that
  \[
    \int_K|\nabla_T u|^2|x|^{\alpha-2}\dd x
    \le\frac{2}{3-\alpha}\sum_{i=1}^n\int_K u_{ii}^2|x|^\alpha\dd x.
  \]

  \item If $-n<\alpha<0$, from \eqref{eq-slice-negative} we directly have
  \[
    \int_K|\nabla u|^2|x|^{-a-2}\dd x
    \le \frac{1}{a}\sum_{i=1}^n\int_K u_{ii}^2|x|^{-a}\dd x.
  \]
  \end{enumerate}
\end{proof}

\subsection{Nonnegativity, rigidity, and strict local log-concavity}\
\label{sec-nonnegative-logconcavity}

For $W_{\alpha}=-\alpha\log|x|$, we define
\begin{equation}\label{eq-I-alpha}
  I_\alpha(u)
  =\int_K\Bigl(
    \|D^2u\|_{\HS}^2
    +D^2 W_{\alpha}(\nabla u,\nabla u)
  \Bigr)\dd\tilde{\mu}_\alpha . 
\end{equation}

\begin{prop}\label{prop-energy}
  Let $K\in\mathcal{K}_{+}^n$ be unconditional, and let $u$ be
  the unconditional Neumann solution from Lemma~\ref{lem-neumann}. If
  $-n<\alpha\le1$, then $I_\alpha(u)\ge0$.
  Moreover, if $I_\alpha(u)=0$, then $u$ is constant on $K$, and
  consequently $\varphi = u_\nu = 0$ on $\partial K$.
\end{prop}

\begin{proof}
  Lemma~\ref{lem-expansion} guarantees
  $D^2u\in L^2(K,\tilde{\mu}_\alpha)$, so all integrals are finite.
  Using \eqref{eq-hessian-split} and
  $\|D^2u\|_{\HS}^2=\sum_i u_{ii}^2+2\sum_{i<j}u_{ij}^2$, we have
  \[
    I_\alpha(u)
    =
    \sum_i\int_K u_{ii}^2\,\dd\tilde{\mu}_\alpha
    +2\sum_{i<j}\int_K u_{ij}^2\,\dd\tilde{\mu}_\alpha
    -\alpha\int_K\frac{|\nabla_Tu|^2}{|x|^2}\dd\tilde{\mu}_\alpha
    +\alpha\int_K\frac{|\partial_ru|^2}{|x|^2}\dd\tilde{\mu}_\alpha .
  \]

  \begin{enumerate}[(i)]
  \item \emph{Case $0<\alpha\le1$:}
  By \eqref{eq-tangential-pos}, we obtain
  \[
    I_\alpha(u)
    \ge
    \frac{3(1-\alpha)}{3-\alpha}
    \sum_{i=1}^n\int_K u_{ii}^2\,\dd\tilde{\mu}_\alpha
    + 2\sum_{i<j}\int_K u_{ij}^2\,\dd\tilde{\mu}_\alpha
    + \alpha\int_K\frac{|\partial_ru|^2}{|x|^2}\dd\tilde{\mu}_\alpha
    \ge0.
  \]
   Now suppose $I_\alpha(u)=0$.

  If $0<\alpha<1$, the coefficients above are strictly positive, so
  equality forces $u_{ii}=0$ for all $i$ and $u_{ij}=0$ for $i\neq j$.
  Hence $u$ is affine on $K\setminus\{0\}$. By unconditionality, the linear part must vanish, and hence $u$ is constant.

  If $\alpha=1$, $I_1(u)=0$ gives $u_{ij}=0$ for $i\neq j$ and $\partial_ru=0$.
  Hence $u(r,\theta)=g(\theta)$ on a small ball $B_{r_0}\subset K$.
  Then $L_Wu=C_\varphi$ reduces to $r^{-2}\Delta_{\mathbb{S}^{n-1}}g=C_\varphi$.
  Comparing two radii yields $C_\varphi=0$ and $\Delta g=0$, so $g$ is constant.
  Since $\partial_ru=0$, $u$ is constant in $K$.

  \item \emph{Case $-n<\alpha<0$:}
  By \eqref{eq-nabla-neg}, we get
  \[
    I_\alpha(u)
    \ge
    -2\alpha\int_K\frac{|\nabla_Tu|^2}{|x|^2}\dd\tilde{\mu}_\alpha
    +2\sum_{i<j}\int_Ku_{ij}^2\,\dd\tilde{\mu}_\alpha
    \ge0.
  \]
    If $I_\alpha(u)=0$, equality forces $\nabla_Tu=0$ and $u_{ij}=0$ for
  $i\neq j$. The condition $\nabla_Tu=0$ makes $u$ radial,
  $u=F(r)$. For a radial function,
  \[
    u_{ij}=\left(F''(r)-\frac{F'(r)}r\right)\frac{x_ix_j}{r^2},\qquad i\ne j.
  \]
  Hence $F''=F'/r$, so $u=A+B|x|^2$.
  Substituting into \eqref{eq-I-alpha} gives
  \[
    I_\alpha(u)=4B^2(n+\alpha)\int_K|x|^\alpha\dd x,
  \]
  where $n+\alpha>0$.  Since $I_\alpha(u)=0$ and
  $\tilde{\mu}_\alpha(K)>0$, we get $B=0$, so $u$ is constant.
  
  \item \emph{Case $\alpha=0$:}
  Then $I_0(u)=\int_K\|D^2u\|_{\HS}^2\dd x\ge0$.

  If $I_0(u)=0$, then $D^2u=0$ a.e., so $u$ is affine on
  $K\setminus\{0\}$. Unconditionality implies that the linear part vanishes, so $u$ is constant.
  \end{enumerate}

  We complete the proof.
\end{proof}

Let $-n<\alpha\le 1$. Let $K$ be an unconditional convex body in $\mathcal{K}_{+}^n$ with support function $h=h_K$. Consider an unconditional smooth perturbation with support function
\[
  h_s=h+s\psi,
  \qquad \psi\in C_{und}^2(\mathbb S^{n-1}),
\]
where $C^k_{und}(\mathbb{S}^{n-1})$ denotes the space of $C^k$ functions on $\mathbb{S}^{n-1}$ that are invariant under all coordinate reflections. Let $K_s$ be the corresponding convex body for $|s|$ sufficiently
small. The normal speed of the moving boundary at $s=0$ is $ \varphi=\psi\circ\nu_K$, where $\nu_K:\partial K\to\mathbb S^{n-1}$ is the Gauss map. Set
\[
  \tilde{\mu}(s)=\tilde{\mu}_\alpha(K_s)
  =\int_{K_s}|x|^\alpha\,\dd x .
\]

\begin{prop}\label{prop-local-logconcave-strict} 
  Under the above assumptions,
  \[
    \left.\frac{\dd^2}{\dd s^2}\right|_{s=0}
    \log \tilde{\mu}(s)\le 0 .
  \]
  Moreover, if $\psi\not\equiv 0$, the inequality is strict.
\end{prop}

\begin{proof}
  Since $\psi\in C^2_{und}(\mathbb{S}^{n-1})$ and $\nu_K$ is $C^1$, we
  have $\varphi\in C^1(\partial K)$. As noted after
  Lemma~\ref{lem-neumann}, the Neumann solution exists and satisfies
  the same conclusions as in the smooth case.

  It suffices to show $\tilde{\mu}(0)\tilde{\mu}''(0)\le\tilde{\mu}'(0)^2$.
  By Lemma~\ref{lem-variation},
  \[
    \tilde{\mu}(0)=\tilde{\mu}_\alpha(K),\quad\tilde{\mu}'(0)=C_\varphi\tilde{\mu}_\alpha(K),
  \]
  and
  \[
  \tilde{\mu}''(0)=\int_{\partial K}\bigl(H_{W_\alpha}\varphi^2-\mathrm{II}^{-1}(\nabla_{\partial K}\varphi,\nabla_{\partial K}\varphi)\bigr)\,d\tilde{\mu}_{\alpha,\partial K}.
  \]
  Applying the singular weighted Reilly formula \eqref{eq-reilly-final} and Proposition \ref{prop-energy}, we obtain
   \begin{align*}
    \frac{\tilde{\mu}'(0)^2}{\tilde{\mu}(0)}
    &=
    I_\alpha(u)
    +
    \int_{\partial K}
      H_{W_\alpha}\varphi^2\,
      \dd\tilde{\mu}_{\alpha,\partial K}  \\
    &\quad
    +
    \int_{\partial K}
      \mathrm{II}(\nabla_{\partial K}u,\nabla_{\partial K}u)\,
      \dd\tilde{\mu}_{\alpha,\partial K}
    -
    2\int_{\partial K}
      \big\langle
        \nabla_{\partial K}\varphi,
        \nabla_{\partial K}u
      \big\rangle
      \dd\tilde{\mu}_{\alpha,\partial K}  \\
    &\ge
    \int_{\partial K}
      H_{W_\alpha}\varphi^2\,
      \dd\tilde{\mu}_{\alpha,\partial K}
    +
    \int_{\partial K}
      \mathrm{II}(\nabla_{\partial K}u,\nabla_{\partial K}u)\,
      \dd\tilde{\mu}_{\alpha,\partial K}       \\
    &\quad
    -
    2\int_{\partial K}
      \big\langle
        \nabla_{\partial K}\varphi,
        \nabla_{\partial K}u
      \big\rangle
      \dd\tilde{\mu}_{\alpha,\partial K}.
  \end{align*}
  Using $\mathrm{II}>0$ for $K\in\mathcal{K}_{+}^n$, and applying the Cauchy-Schwarz inequality, we deduce
    \[
    \frac{\tilde{\mu}'(0)^2}{\tilde{\mu}(0)}
    \ge
    \int_{\partial K}
    \left(
      H_{W_\alpha}\varphi^2
      -
      \mathrm{II}^{-1}
      \bigl(
        \nabla_{\partial K}\varphi,
        \nabla_{\partial K}\varphi
      \bigr)
    \right)
    \dd\tilde{\mu}_{\alpha,\partial K}
    =\tilde{\mu}''(0).
  \]
  Moreover, if $\psi\not\equiv0$ and the inequality is not strict, then the equality implies $I_\alpha(u)=0$. Proposition~\ref{prop-energy} then forces $\varphi = u_\nu = 0$ on $\partial K$. Since $\varphi = \psi \circ \nu_K$ and the Gauss map $\nu_K$ is onto $\mathbb{S}^{n-1}$, we obtain $\psi \equiv 0$, a contradiction. Hence the inequality is strict.
\end{proof}

\subsection{Proof of Theorem~\ref{thm-unconditional}}
\label{sec-proof-unc-BM}

\begin{proof}[Proof of Theorem~\ref{thm-unconditional}]
  Put \(q=n+\alpha\). Then \(\alpha=q-n\in(-n,1]\), and $ \tV_q(\cdot)=\frac qn\,\tilde{\mu}_\alpha(\cdot).$
  Hence it suffices to prove the Brunn-Minkowski inequality for
  \(\tilde{\mu}_\alpha\).

  \emph{Smooth case and equality.}
  Assume first that \(K,L\) are unconditional \(C^2_+\) convex bodies.
  For \(t\in[0,1]\), set $K_t=(1-t)K+tL$.
  Each \(K_t\) is again unconditional and \(C^2_+\).  Fix \(t_0\in(0,1)\).
  In a neighbourhood of \(t_0\),
  \[
    h_{K_t}=h_{K_{t_0}}+(t-t_0)(h_L-h_K),
  \]
  so \(K_t\) is a smooth unconditional perturbation of \(K_{t_0}\).  Proposition~\ref{prop-local-logconcave-strict} therefore
  gives
  \[
    \frac{\mathrm d^2}{\mathrm dt^2}\Big|_{t=t_0}
    \log\tilde{\mu}_\alpha(K_t)\le0.
  \]
  Since \(t_0\in(0,1)\) is arbitrary,
  \(\log\tilde{\mu}_\alpha(K_t)\) is concave on \([0,1]\). Hence
  \begin{equation}\label{ineq-mu-0}
       \tilde{\mu}_\alpha((1-t)K+tL)
    \ge
    \tilde{\mu}_\alpha(K)^{1-t}\tilde{\mu}_\alpha(L)^t.
  \end{equation}
  
   Let $a=\tilde{\mu}_\alpha(K)^{\frac1q}$, $b=\tilde{\mu}_\alpha(L)^{\frac1q}$, and $t_0=\frac{b}{a+b}$. Let $\widehat K=\frac{K}{a}$ and $\widehat L=\frac{L}{b}$, so that
  \[
    \tilde{\mu}_\alpha(\widehat K)=\tilde{\mu}_\alpha(\widehat L)=1.
  \]
  Applying \eqref{ineq-mu-0} to \(\widehat K,\widehat L\) gives
  \[
    \tilde{\mu}_\alpha\!\left(\frac{K+L}{a+b}\right)
    =
    \tilde{\mu}_\alpha((1-t_0)\widehat K+t_0\widehat L)
    \ge 1.
  \]
  By the \((n+\alpha)\)-homogeneity of \(\tilde{\mu}_\alpha\), we obtain
  \[
    \tilde{\mu}_\alpha(K+L)^{\frac1q}
    \ge
    \tilde{\mu}_\alpha(K)^{\frac1q}
    +
    \tilde{\mu}_\alpha(L)^{\frac1q},
  \]
    which implies the desired inequality \eqref{unc-BM} for \(C^2_+\) unconditional bodies.

    Moreover, suppose equality holds in \eqref{unc-BM}. This is equivalent to
  \[
    \tilde{\mu}_\alpha((1-t_0)\widehat K+t_0\widehat L)=1.
  \]
  Define $G(t)=\log\tilde{\mu}_\alpha((1-t)\widehat K+t\widehat L)$ for $t\in[0,1]$.  As shown above, \(G\) is concave, and $G(0)=G(\lambda)=G(1)=0$. Hence \(G\equiv0\) on \([0,1]\). 
  If \(\widehat K\neq\widehat L\),
  Proposition~\ref{prop-local-logconcave-strict} gives
  \(G''(t_0)<0\) for every \(t_0\in(0,1)\), which is a contradiction.
  Therefore \(\widehat K=\widehat L\), so \(K\) and \(L\) are dilates of each other.
  The converse is immediate from the homogeneity of \(\tV_q\).
  
  \emph{Approximation.}
  For arbitrary unconditional convex bodies \(K,L\), standard
  approximation yields sequences \(K_j,L_j\) of unconditional \(C^\infty_+\)
  convex bodies such that
  \[
    K_j\to K,\qquad L_j\to L,\qquad
    K_j+L_j\to K+L
  \]
  in the Hausdorff metric, and
  \[
    \tilde{\mu}_\alpha(K_j)\to\tilde{\mu}_\alpha(K),\qquad
    \tilde{\mu}_\alpha(L_j)\to\tilde{\mu}_\alpha(L),\qquad
    \tilde{\mu}_\alpha(K_j+L_j)\to\tilde{\mu}_\alpha(K+L).
  \]
 
  Indeed, one may convolve the support function on \(O(n)\), symmetrize by coordinate reflections, and add a small Euclidean ball; the convergence of $\tilde{\mu}_\alpha$ follows from dominated convergence for \(\alpha\ge0\) and from polar coordinates for \(\alpha<0\). 
  Applying the smooth case to \(K_j,L_j\) and passing to
  the limit proves \eqref{unc-BM} for all unconditional convex bodies.

  The proof of Theorem~\ref{thm-unconditional} is complete.
\end{proof}

\section{Uniqueness consequences}
\label{sec-8}

In this final section, we derive uniqueness results for dual curvature
measures as applications of the Brunn-Minkowski inequalities established
in the previous sections.

Let $\tC_q(K,\cdot)$ be the $q$-th dual curvature measure of $K$ and,
for $p\in\R$, let
\[
  d\tC_{p,q}(K,u)=h_K(u)^{-p}\,d\tC_q(K,u)
\]
be the $(p,q)$-th dual curvature measure introduced by Lutwak-Yang-Zhang
\cite{LYZ18}. When $p=1$ and $K$ is strictly convex, the variational
formula gives
\begin{equation}\label{variation-dual-curvature}
  \left.\frac{d}{dt}\right|_{t=0}\tV_q(K_t)
  =q\int_{\mathbb S^{n-1}} f(u)\,d\tC_{1,q}(K,u),
\end{equation}
where $h_{K_t}=h_K+t f$ and $f\in C^2(\mathbb S^{n-1})$. See \cite[Theorem~3.1]{HLYZ16} or \cite[Section~3]{XZ22}.

A standard consequence of a Brunn-Minkowski inequality is the
corresponding Minkowski inequality.  Since the derivation is short and
will be used explicitly in the uniqueness proof, we include it here
together with the identification of the equality case.

\begin{lem}\label{lem-Minkowski}
  Let $0<q\le n+1$ and let $K,L$ be unconditional $C^2_+$ convex bodies.
  Then
  \begin{equation}\label{Mink-unc}
    \int_{\mathbb S^{n-1}} h_L(u)\,d\tC_{1,q}(K,u)
    \ge \tV_q(K)^{1-\frac1q}\,\tV_q(L)^{\frac1q},
  \end{equation}
  with equality if and only if $K$ and $L$ are dilates of each other.
\end{lem}

\begin{proof}
  For $t\ge0$ set $K_t=K+tL$.  Its support function is
  $h_{K_t}=h_K+t h_L$, so we may apply the variational formula
  \eqref{variation-dual-curvature} with $f=h_L$.  Because
  $\frac{d}{dt}\tV_q(K_t)=q\int h_L\,d\tC_{1,q}(K_t)$, evaluating at
  $t=0$ gives
  \[
    \tV_q'(0):=\left.\frac{d}{dt}\right|_{t=0}\tV_q(K_t)
    =q\int_{\mathbb S^{n-1}} h_L\,d\tC_{1,q}(K).
  \]
  The Brunn-Minkowski inequality
  $\tV_q(K_t)^{\frac1q}\ge\tV_q(K)^{\frac1q}+t\tV_q(L)^{\frac1q}$
  (Theorem~\ref{thm-unconditional}) implies that the function
  $F(t)=\tV_q(K_t)^{\frac1q}$ satisfies
  $F(t)-F(0)\ge t\tV_q(L)^{\frac1q}$.  Since $K$ is $C^2_+$, $F$ is
  differentiable at $0$ and
  $F'(0)=\frac1q\tV_q(K)^{\frac1q-1}\tV_q'(0)$.  Hence
  \[
    \frac1q\tV_q(K)^{\frac1q-1}\cdot q\int h_L\,d\tC_{1,q}(K)
    \ge\tV_q(L)^{\frac1q},
  \]
  which is precisely \eqref{Mink-unc}.

  Now suppose equality holds in \eqref{Mink-unc}.  Then
  $F'(0)=F(1)-F(0)$, i.e.\ the linear bound for $F$ is attained at the
  derivative.  Because $F$ is concave on $[0,1]$ (a consequence of the
  Brunn-Minkowski inequality; see the proof of
  Theorem~\ref{thm-unconditional}), the only way this can happen is
  that $F$ is affine on $[0,1]$.  The equality case characterization
  of Theorem~\ref{thm-unconditional} (established in the proof above)
  forces $K$ and $L$ to be dilates.
  Conversely, if $L=\lambda K$, one verifies directly that both sides
  of \eqref{Mink-unc} equal $\lambda\tV_q(K)$, so equality holds.
\end{proof}

\begin{rmk}
  The same argument, using Theorem~\ref{thm-endpoint} instead of
  Theorem~\ref{thm-unconditional}, yields the Minkowski inequality
  \begin{equation}\label{Mink-endpoint}
    \int_{\mathbb S^{n-1}} h_L(u)\,d\tC_{1,n+2}(K,u)
    \ge \tV_{n+2}(K)^{1-\frac1{n+2}}\,\tV_{n+2}(L)^{\frac1{n+2}}
  \end{equation}
  for origin‑symmetric $C^2_+$ convex bodies $K,L$, again with equality
  exactly for dilates.
\end{rmk}

We also need the elementary identity that expresses the total mass of
$\tC_{1,q}$.

\begin{lem}\label{lem-totalmass}
  For any $K\in\cK_+^n$ and any $q>0$,
  \begin{equation}\label{totalmass}
    \int_{\mathbb S^{n-1}} h_K(u)\,d\tC_{1,q}(K,u)=\tV_q(K).
  \end{equation}
\end{lem}

\begin{proof}
  Scale $K$ by setting $K_t=(1+t)K$, so that $h_{K_t}=(1+t)h_K$.
  On one hand, $\tV_q(K_t)=(1+t)^q\tV_q(K)$; differentiating at $t=0$
  gives $\tV_q'(0)=q\tV_q(K)$.  On the other hand, the variational
  formula \eqref{variation-dual-curvature} with $f=h_K$ yields
  $\tV_q'(0)=q\int h_K\,d\tC_{1,q}(K)$.  Equating the two expressions
  proves \eqref{totalmass}.
\end{proof}

\begin{cor}\label{cor-uniqueness-1}
  Let $0<q\le n+1$ and let $K,L$ be unconditional $C^2_+$ convex
  bodies.  If
  \[
    \tC_{1,q}(K,\cdot)=\tC_{1,q}(L,\cdot),
  \]
  then $K=L$ unless $q=1$, in which case $K$ and $L$ are dilates of each other.
\end{cor}

\begin{proof}
  From Lemma~\ref{lem-Minkowski} we have the two inequalities
  \[
    \int h_L\,d\tC_{1,q}(K) \ge \tV_q(K)^{1-\frac1q}\tV_q(L)^{\frac1q}, \qquad
    \int h_K\,d\tC_{1,q}(L) \ge \tV_q(L)^{1-\frac1q}\tV_q(K)^{\frac1q}.
  \]
  Set $A=\tV_q(K)^{1/q}$ and $B=\tV_q(L)^{1/q}$.  By the assumption
  $\tC_{1,q}(K)=\tC_{1,q}(L)$ and Lemma~\ref{lem-totalmass}, the
  left-hand sides equal $B^q$ and $A^q$, respectively. Hence
  \[
    B^q \ge A^{q-1}B,\qquad A^q \ge B^{q-1}A.
  \]
  For every $q\neq1$, dividing by $A,B>0$ gives
  $B^{q-1}=A^{q-1}$, and since the power function is strictly
  monotone on $(0,\infty)$, we obtain $A=B$. Therefore equality holds
  in both Minkowski inequalities, and the equality case of
  Lemma~\ref{lem-Minkowski} yields $L=\lambda K$ for some
  $\lambda>0$.  Substituting into $\tC_{1,q}(K)=\tC_{1,q}(\lambda K)$
  and using the homogeneity
  $d\tC_{1,q}(\lambda K,u)=\lambda^{q-1}\,d\tC_{1,q}(K,u)$ gives  $\lambda^{q-1}=1$, hence $\lambda=1$. Thus $K=L$.

  If $q=1$, the two inequalities above become identities and give no
  information about $A$ and $B$.  However, by the first Minkowski
  inequality together with the assumption and Lemma~\ref{lem-totalmass}, equality must hold in that Minkowski inequality, so its equality case gives that $K$ and $L$ are dilates of each other.
  Since $\tC_{1,1}$ is dilation-invariant, this is optimal and the
  dilation factor is arbitrary.
\end{proof}

\begin{cor}\label{cor-uniqueness-endpoint}
  Let $K,L$ be origin‑symmetric $C^2_+$ convex bodies and suppose
  \[
    \tC_{1,n+2}(K,\cdot)=\tC_{1,n+2}(L,\cdot).
  \]
  Then $K=L$.
\end{cor}

\begin{proof}
  The argument is identical to that of
  Corollary~\ref{cor-uniqueness-1}, now using the Minkowski inequality
  \eqref{Mink-endpoint} together with Lemma~\ref{lem-totalmass} for $q=n+2$.
\end{proof}

\begin{rmk}
  More general uniqueness statements for $(p,q)$-th dual curvature measures can be obtained by combining the $L_p$ variational formula
  with the corresponding $L_p$ Brunn-Minkowski inequalities.  In the range $p\ge q$ such uniqueness results were established by Xi-Zhang \cite{XZ22}.  
  In contrast, for the case $p=0$, Xiong and Yang
  \cite[Theorem~1.2]{XiongY26} proved the non-uniqueness results for origin-symmetric convex bodies when $q>n$.
\end{rmk}

\begin{ack}
	  H. Li and Y. Zhang were supported by NSFC grant No. 12471047.
    Y. Wan was supported by Hong Kong RGC grant (Early Career Scheme) of Hong Kong No. 24304222 and No. 14300623, and NSFC grant No. 12222122.
\end{ack}

\end{document}